\documentclass[11pt,reqno,oneside]{amsart}

\usepackage{graphicx}
\usepackage{amssymb}
\usepackage{epstopdf}
\usepackage{cases}
\usepackage[pagewise]{lineno}

\usepackage[a4paper, total={6in, 9in}]{geometry}

\usepackage{amsmath,amsfonts,amsthm,mathrsfs,amssymb,cite}
\usepackage[usenames]{color}

\newtheorem{thm}{Theorem}[section]

\newtheorem{cor}[thm]{Corollary}
\newtheorem{lem}{Lemma}[section]

\theoremstyle{definition}
\newtheorem{defn}{Definition}[section]

\theoremstyle{remark}

\newtheorem{rem}{Remark}[section]
\numberwithin{equation}{section}

\numberwithin{equation}{section}

\usepackage{tikz}
\allowdisplaybreaks

\newcommand{\R}{\mathbb{R}} \newcommand{\mathR}{\mathbb{R}}
\newcommand{\Rn}{{\mathR^n}}

\newcommand{\calF}{\mathcal{F}}

\newcommand{\scrS}{\mathscr{S}}
\newcommand{\calC}{\mathcal{C}}

\newcommand{\supp}{\mathop{\rm supp}}
\newcommand{\dist}{\mathop{\rm dist}}

\newcommand{\Forall}{{~\forall\,}}
\newcommand{\Exists}{{~\exists\,}}
\newcommand{\st}{\textrm{~s.t.~}}
\newcommand{\aee}{\textrm{~a.e.~}}
\newcommand{\ass}{\textrm{~a.s.~}}

\newcommand{\dif}[1]{\,\mathrm{d}{#1}} 
\newcommand{\nrm}[2][]{ \| {#2} \|_{#1}}
\newcommand{\agl}[1][\cdot]{ \langle {#1} \rangle}

\usepackage{tikz}
\newcommand{\Rk}{{\mathcal{R}_{k}}}

\newcommand{\sq}[1]{{#1}}

\newcounter{saveeqn}

\title[Determining a random Schr\"odinger operator]{Determining a random Schr\"odinger operator: both potential and source are random}

\author{Jingzhi Li}
\address{Department of Mathematics, Southern University of Science and Technology, Shenzhen, China}
\email{li.jz@sustech.edu.cn}

\sq{
\author{Hongyu Liu}
\address{Department of Mathematics, City University of Hong Kong, Hong Kong SAR, China}
\email{hongyu.liuip@gmail.com, hongyliu@cityu.edu.hk}
}

\author{Shiqi Ma}
\address{Department of Mathematics and Statistics, University of Jyv\"askyl\"a, Finland}
\email{mashiqi01@gmail.com, shiqi.s.ma@jyu.fi}

\begin{document}

\begin{abstract}
	
	We study an inverse scattering problem associated with a Schr\"odinger system where both the potential and source terms are random and unknown. 
	The well-posedness of the forward scattering problem is first established in a proper sense. 
	We then derive two unique recovery results in determining the rough strengths of the random source and the random potential, by using the corresponding far-field data. 
	The first recovery result shows that a single realization of the passive scattering measurements uniquely recovers the rough strength of the random source. 
	The second one shows that, by a single realization of the backscattering data, the rough strength of the random potential can be recovered. 
	The ergodicity is used to establish the single realization recovery.
	The asymptotic arguments in our study are based on techniques from the theory of pseudodifferential operators and microlocal analysis.
	
	\medskip

	\noindent{\bf Keywords:}~~
	inverse scattering, random source and medium, ergodicity, pseudodifferential operators, microlocal analysis
	
	{\noindent{\bf 2010 Mathematics Subject Classification:}~~35Q60, 35J05, 31B10, 35R30, 78A40}
	
\end{abstract}

\maketitle

\section{Introduction} \label{sec:intro-MLmGWsSchroEqu2019}

\subsection{Mathematical formulations} \label{subsec:formu-MLmGWsSchroEqu2019}

In this paper, we are mainly concerned with the following random Schr\"odinger system
\begin{subequations} \label{eq:1-MLmGWsSchroEqu2019}
	\begin{numcases}{}
	\displaystyle{ \big( -\Delta - E + q(x,\omega) \big) u(x, \sqrt{E}, d, \omega)= f(x,\omega), \quad x\in\mathbb{R}^3, } \label{eq:1a-MLmGWsSchroEqu2019} \medskip\\
	\displaystyle{ u(x, \sqrt{E}, d, \omega)=\alpha e^{\mathrm{i}\sqrt E x\cdot d}+u^{sc}(x, \sqrt{E}, d,\omega), } \label{eq:1b-MLmGWsSchroEqu2019} \medskip\\
	\displaystyle{ \lim_{r\rightarrow\infty} r\left(\frac{\partial u^{sc}}{\partial r}-\mathrm{i}\sqrt{E} u^{sc} \right)=0,\quad r:=|x|, } \label{eq:1c-MLmGWsSchroEqu2019}
	\end{numcases}
\end{subequations}
where \sq{$\mathrm{i} := \sqrt{-1}$,} and $\omega$ in \eqref{eq:1a-MLmGWsSchroEqu2019} is a random sample belonging to $\Omega$ with $(\Omega,\mathcal F,\mathbb P)$ being a complete probability space, 
and $f(x,\omega)$ and $q(x,\omega)$ are independently distributed generalized Gaussian random fields with zero-mean and are supported in bounded domains $D_f$ and $D_q$, respectively.
$E\in\mathbb{R}_+$ is the energy level. 
In the sequel, we follow the convention to replace $E$ with $k^2$, namely $k := \sqrt{E} \in \R_+$, which can be understood as the wave number. 
In \eqref{eq:1b-MLmGWsSchroEqu2019}, ${d \in \mathbb{S}^2:=\{ x \in \R^3 \,;\, |x| = 1 \}}$ signifies the incident direction of the plane wave, and $\alpha$ takes the value of either 0 or 1 to impose or suppress the incident wave, respectively. 
$u^{sc}$ in \eqref{eq:1b-MLmGWsSchroEqu2019} is the scattered wave field, which is also random due to the randomness of the source and potential.
The limit \eqref{eq:1c-MLmGWsSchroEqu2019} is the Sommerfeld Radiation Condition (SRC) \cite{colton2012inverse} that characterizes the outgoing nature of the scattered field $u^{sc}$. 
The random system \eqref{eq:1-MLmGWsSchroEqu2019} describes the quantum scattering \cite{eskin2011lectures,griffiths2016introduction} associated with a  source $f$ and a potential $q$ at the energy level $k^2$.

$f$ and $q$ in equation \eqref{eq:1a-MLmGWsSchroEqu2019} are assumed to be generalized Gaussian random fields.
It means that $f$ is a random distribution and the mapping 
\[\omega \in \Omega \ \mapsto \ \agl[f(\cdot,\omega),\varphi] \in \mathbb C\] 
is a Gaussian random variable whose probabilistic measure depends on the test function $\varphi$. 
The same notation applies to $q$.
There are different types of generalized Gaussian random fields \cite{rozanov1982markov}.
In our setting, we assume that $f$ and $q$ are two \emph{microlocally isotropic} generalized Gaussian random (\emph{m.i.g.r.}~for short) functions/distributions; see Definition 2.1 in the following.
The m.i.g.r.~model has been under intensive studies; see, e.g., \cite{Lassas2008,LassasA,caro2016inverse,LiHelinLiinverse2018}.
Two important parameters of a m.i.g.r.~distribution are its \emph{rough order} and \emph{rough strength}. Roughly speaking, the rough order, which is a real number, determines the degree of spatial roughness of the m.i.g.r.~distribution, and the rough strength, which is a real-valued function, indicates its spatial correlation length and intensity.
The rough strength also captures the micro-structure of the object in interest \cite{Lassas2008}.
We shall give a more detailed introduction to this random model in Section \ref{subsec:RandomModel-MLmGWsSchroEqu2019}.

In this work, 
we denote the rough order of $f$ (resp.~$q$) as $-m_f$ (resp.~$-m_q$), and the rough strength as $\mu_f(x)$ (resp.~$\mu_q(x)$).
The main purpose of this work is to recover the rough strengths of both the source and the potential using either passive or active far-field measurements associated with \eqref{eq:1-MLmGWsSchroEqu2019}.

\subsection{Statement of the main results} \label{subsec:MainRst-MLmGWsSchroEqu2019}

In order to study the inverse scattering problem, i.e., the recovery of $\mu_f$ and $\mu_q$, we first need to have a thorough understanding of the direct scattering problem.
For the case where both the source and the potential are deterministic and $L^\infty$ functions with compact supports, the well-posedness of the direct problem of system \eqref{eq:1-MLmGWsSchroEqu2019} is known; see, e.g., \cite{colton2012inverse, eskin2011lectures, mclean2000strongly}.
Moreover, there holds the following asymptotic expansion of the outgoing radiating field $u^{sc}$ as $|x| \to +\infty$,
\begin{equation*} 
u^{sc}(x) = \frac{e^{\mathrm{i}k |x|}}{|x|} u^\infty(\hat x, k, d) + \mathcal{O} \left( \frac{1}{|x|^2} \right).
\end{equation*}
$u^\infty(\hat x, k, d)$ is referred to as the far-field pattern, which encodes information of the potential and the source.
$\hat x:=x/|x|$ and $d$ in $u^\infty(\hat x, k, d)$ are unit vectors and they respectively stand for the observation direction and the impinging direction of the incident wave.
When $d = -\hat x$, $u^\infty(\hat x, k, -\hat x)$ is called the backscattering far-field pattern.

In the random setting, however, due to the randomness inherited in the source and potential terms, the regularities of the corresponding scattering wave field are much worse \cite{caro2016inverse,Lassas2008}. This makes those standard PDE theories invalid for the direct problem of system \eqref{eq:1-MLmGWsSchroEqu2019}.
To tackle this issue, we shall reformulate the direct problem and show that the direct problem is still well-posed in a proper sense. Therefore, our direct problem can be formulated as
\begin{equation*} 
	\displaystyle{ (f, q) \rightarrow \{u^{sc}(\hat x, k, d, \omega), u^\infty(\hat x, k, d, \omega) \,;\, \omega \in \Omega,\, \hat x \in \mathbb{S}^2, k \in \R_+,\,  d \in \mathbb{S}^2 \}. }
\end{equation*}

The well-posedness of the direct scattering problem enables us to explore our inverse problems.
Due to the fact that the precise values of a random function provide little information about its statistical properties,
we are interested in the recovery of the 
rough strengths of the source and the potential by knowledge of the far-field patterns.

In the recovery procedure, we recover $\mu_f$ and $\mu_q$ in a \emph{sequential} way by knowledge of the associated far-field pattern measurements $u^\infty(\hat x, k, d, \omega)$.
By sequential, we mean that $\mu_f$ and $\mu_q$ are recovered by the corresponding far-field data sets one-by-one.
In addition to this, in the recovery procedure, both the \emph{passive} and \emph{active} measurements are utilized. 
When $\alpha = 0$, the incident wave is suppressed and the scattering is solely generated by the unknown source. The corresponding far-field pattern is referred to as the passive measurement.
In this case, the far-field pattern is independent of the incident direction $d$, and we denote it as $u^\infty(\hat x, k, \omega)$.  	
When $\alpha = 1$, the scattering is generated by both the active source and the incident wave, and the far-field pattern is referred to as the active measurement, and is denoted as $u^\infty(\hat x, k, d, \omega)$.

With the above discussion, our inverse problems can be formulated as
\begin{equation*} 
\displaystyle{ \left\{
\begin{aligned}
\mathcal M_f(\omega) := & \ \{\, u^\infty(\hat x, k, \omega) \,;\, \forall \hat{x} \in \mathbb{S}^2,\, \forall k \in \R_+\, \} && \rightarrow \quad \mu_f, \\
\mathcal M_q(\omega) := & \ \{\, u^\infty(\hat x, k, -\hat x, \omega) \,;\, \forall \hat{x} \in \mathbb{S}^2,\, \forall k \in \R_+\, \} && \rightarrow \quad \mu_q.
\end{aligned}
\right. }
\end{equation*}
The data set $\mathcal M_f(\omega)$ (abbr.~$\mathcal M_f$) corresponds to the passive measurement ($\alpha = 0$), while the data set $\mathcal M_q(\omega)$ (abbr.~$\mathcal M_q$) corresponds to the active measurement ($\alpha = 1$).
Different random samples $\omega$ generate different data sets.
Our study shows that in certain general scenarios the data sets $\mathcal M_f(\omega)$, $\mathcal M_q(\omega)$ with a fixed $\omega\in\Omega$ can uniquely recover $\mu_f$ and $\mu_q$, respectively.


With the potential term being unknown, the inverse source problem, i.e., the recovery of $\mu_f$, becomes highly nonlinear and thus more challenging. 
One possibility to tackle this situation is to put some geometrical assumption on the locations of the source and the potential. 
In what follows, we assume that there is a positive distance between the convex hulls of the supports of $f$ and $q$, i.e.,
\begin{equation} \label{eq:fqSeparation-MLmGWsSchroEqu2019}
	\dist(\mathcal{CH}(D_f), \mathcal{CH}(D_q)) := \inf\{\, |x - y| \,;\, x \in \mathcal{CH}(D_f),\, y \in \mathcal{CH}(D_q) \,\} > 0,
\end{equation} 
where $\mathcal{CH}$ means taking the convex hull of a domain. Therefore, one can find a plane which separates $D_f$ and $D_q$. In what follows, in order to simplify the exposition, we assume that $D_f$ and $D_q$ are convex domains and hence $\mathcal{CH}(D_f)=D_f$ and $\mathcal{CH}(D_q)=D_q$. Moreover, we let $\boldsymbol{n}$ denote the unit normal vector of the aforementioned plane that separates $D_f$ and $D_q$, pointing from the half-space containing $D_f$ into the half-space containing $D_q$.

In system \eqref{eq:1-MLmGWsSchroEqu2019}, both the source and the potential are assumed to be unknown. 
Moreover, the source and the potential are generalized random functions of the same type. 
These issues make the decoupling of $\mu_f$ and $\mu_q$ far more difficult. 
However, some a-priori information about the rough orders of $f$ and $q$ can help us to achieve the recoveries.
Now we are ready to state our main recovery results of the inverse problems.

\begin{thm} \label{thm:UniSource-MLmGWsSchroEqu2019}
	Suppose that $f$ and $q$ in system \eqref{eq:1-MLmGWsSchroEqu2019} are m.i.g.r.~distributions of order $-m_f$ and $-m_q$, respectively, satisfying 
	\begin{equation} \label{eq:mqmf-MLmGWsSchroEqu2019}
	\sq{2 < m_f < 4, \ m_f < 5m_q - 11.}
	\end{equation}
	Assume that \eqref{eq:fqSeparation-MLmGWsSchroEqu2019} is satisfied and $\boldsymbol{n}$ is defined as above. Then, independent of $\mu_q$, the data set $\mathcal M_f(\omega)$ for a fixed $\omega\in\Omega$ can uniquely recover $\mu_f$ almost surely. Moreover, the recovering formula is given by
	\begin{equation} \label{eq:Thm1Rec-MLmGWsSchroEqu2019}
	\displaystyle{ \widehat \mu_f(\tau \hat x)
		= \left\{\begin{aligned}
		& \lim_{K \to +\infty} \frac {4\sqrt{2\pi}} K \int_K^{2K} k^{m_f} \overline{u^\infty(\hat x,k,\omega)} u^\infty(\hat x,k+\tau,\omega) \dif k, \quad \hat{x} \cdot \boldsymbol{n} \geq 0,\medskip \\
		& \ \overline{\widehat \mu_f(-\tau \hat x)}, \quad \hat{x} \cdot \boldsymbol{n} < 0,
		\end{aligned}\right. }
	\end{equation}
	where $\tau \geq 0$ and $u^\infty(\hat x,k,\omega) \in \mathcal M_f(\omega)$.
\end{thm}


\begin{rem}\label{rem:1.1}
	In Theorem \ref{thm:UniSource-MLmGWsSchroEqu2019}, $\mu_f$ can be uniquely recovered without a-priori knowledge of $q$. {Moreover, since $\alpha=0$ in $\mathcal{M}_f(\omega)$, Theorem \ref{thm:UniSource-MLmGWsSchroEqu2019} indicates that $\mu_f$ can be uniquely recovered by a single realization of the passive scattering measurement.}
	Due to the requirement $\hat x \cdot \boldsymbol{n} \geq 0$, only half of all the observation directions are needed.	
	Besides, for the sake of simplicity, we set the wave number $k$ in the definition of $\mathcal M_f$ to be running over all the positive real numbers. 
	But, according to \eqref{eq:Thm1Rec-MLmGWsSchroEqu2019}, it is sufficient to let $k$ be greater than any fixed positive number.
	These remarks also apply to Theorem \ref{thm:UniPot1-MLmGWsSchroEqu2019} in what follows.
	Moreover, it is noted that in the definition of m.i.g.r.~distribution (cf. Definition~2.1), $\mu$ is defined as a real-valued function. Therefore, $\widehat \mu_f$ in Theorem \ref{thm:UniSource-MLmGWsSchroEqu2019} (and $\widehat \mu_q$ in Theorem \ref{thm:UniPot1-MLmGWsSchroEqu2019} below) is a conjugate-symmetric function.
	\sq{It is worth mentioning that the a-priori requirement $2 < m_f < 4$ comes from \eqref{eq:hotF1F1J1-MLmGWsSchroEqu2019}-\eqref{eq:hotF1F1J2-MLmGWsSchroEqu2019} and \eqref{eq:F1IntFinit-MLmGWsSchroEqu2019}, while the a-priori requirement $m_f < 5 m_q - 11$ comes from \eqref{eq:F2IntConvInter-MLmGWsSchroEqu2019} in our subsequent analysis.}
\end{rem}

To recover $\mu_q$, the active scattering measurement shall be needed in our recovery procedure.

\begin{thm} \label{thm:UniPot1-MLmGWsSchroEqu2019}
	Under the same condition as that in Theorem \ref{thm:UniSource-MLmGWsSchroEqu2019} \sq{with an additional assumption that $m_q < m_f$,} and independent of $\mu_f$, the data set $\mathcal M_q(\omega)$ for a fixed $\omega \in \Omega$ can uniquely recover $\mu_q$ almost surely. Moreover, the recovering formula is given by
	\begin{equation} \label{eq:Thm2Rec-MLmGWsSchroEqu2019}
	\displaystyle{ \widehat \mu_q(\tau \hat x)
		= \left\{\begin{aligned}
		& \lim_{K \to +\infty} \frac {4\sqrt{2\pi}} K \int_K^{2K} k^{m_q} \overline{u^\infty(\hat x,k,-\hat x,\omega)} u^\infty(\hat x,k\!+\! \tfrac \tau 2,-\hat x,\omega) \dif k, \ \hat{x} \cdot \boldsymbol{n} \geq 0,\medskip \\
		& \ \overline{\widehat \mu_f(-\tau \hat x)}, \quad \hat{x} \cdot \boldsymbol{n} < 0,
		\end{aligned}\right. }
	\end{equation}
	where $\tau \geq 0$ and $u^\infty(\hat x,k,-\hat x,\omega) \in \mathcal M_q(\omega)$.
\end{thm}

\begin{rem}
	It is emphasized that the recovery result in Theorem \ref{thm:UniPot1-MLmGWsSchroEqu2019} is independent of $\mu_f$. Moreover, we only make use of 
	a single realization of the active backscattering measurement. 
	\sq{We would also like to point out that the additional a-priori requirement $m_q < m_f$ comes from \eqref{eq:Lpq1-MLmGWsSchroEqu2019} in our subsequent analysis.}
\end{rem}

\subsection{Discussion and connection to the existing results} \label{subsec:DiscConn-MLmGWsSchroEqu2019}

There is abundant literature for inverse scattering problems associated with either passive or active measurements.  
Given a known potential, the recovery of an unknown source term by the corresponding passive measurement is referred to as the inverse source problem. We refer to \cite{bao2010multi,Bsource,BL2018,ClaKli,GS1,Isakov1990,IsaLu,Klibanov2013,KS1,WangGuo17,Zhang2015} and references therein for both theoretical uniqueness/stability results and computational methods for the inverse source problem in the deterministic setting. 
The simultaneous recovery of an unknown source and its surrounding potential was also investigated in the literature. 
In \cite{KM1,liu2015determining}, motivated by applications in thermo- and photo-acoustic tomography, the simultaneous recovery of an unknown source and its surrounding medium parameter was considered. This type of inverse problems also arise in the magnetic anomaly detections using geomagnetic monitoring \cite{DLL1,DLL2}. The studies in \cite{DLL1, DLL2, KM1,liu2015determining} were confined to the deterministic setting and associated mainly with the passive measurement.
For the random/stochastic case,
the determination of a random source by the corresponding passive measurement was also recently studied in \cite{bao2016inverse,Lu1,Yuan1,LiHelinLiinverse2018}.
In \cite{LiHelinLiinverse2018}, the homogeneous Helmholtz system with a random source is studied.
Compared with \cite{LiHelinLiinverse2018}, system \eqref{eq:1-MLmGWsSchroEqu2019}
in this paper
%
comprises of both unknown source and unknown potential, making the corresponding study radically more challenging.
The determination of a random potential by the corresponding active measurement, with the source term being zero,
was established in \cite{caro2016inverse}. 
We also refer to \cite{LassasA, Lassas2008, Blomgren2002, Borcea2002, Borcea2006} and references therein for more relevant studies on random inverse medium problems.

We are particularly interested in the case with a single realization of the random sample, namely the $\omega$ is fixed in the recovery of the source and potential; 
see the recovery formulae \eqref{eq:Thm1Rec-MLmGWsSchroEqu2019}-\eqref{eq:Thm2Rec-MLmGWsSchroEqu2019}.
In our approach, we assume that the backscattering far-field data $u^\infty(\hat x, k, -\hat x, \omega)$ for different observation directions are generated by a single realization of the random sample \cite{caro2016inverse}.
Intuitively, a particular realization of $f$ or $q$ provides us little information about their statistical properties. 
However, our study indicates that a \emph{single realization} of the far-field measurement can be used to uniquely recover the rough strength in certain scenarios. 
A crucial assumption to make the single realization recovery possible is that the randomness is independent of the wave number $k$. 
Indeed, there are variant applications in which the randomness changes slowly or is independent of time \cite{caro2016inverse, Lassas2008}, and by temporal Fourier transforming into the frequency domain, they actually correspond to the aforementioned situation. 
The single realization recovery has been studied in the literature; 
see, e.g., \cite{caro2016inverse,Lassas2008,LassasA,llm2018random}.
The idea of this article is mainly motivated by \cite{caro2016inverse,llm2018random}.

Compared with our previous work \cite{llm2018random}, the result of this paper has two major differences.
First, the random models are different. 
In \cite{llm2018random}, the random part of the source is assumed to be a Gaussian white noise, 
while in system \eqref{eq:1-MLmGWsSchroEqu2019}, the potential and the source are assumed to be m.i.g.r.~distributions. The m.i.g.r.~distribution can 
fit larger range of randomness 
by tuning its rough order.
Second, in system \eqref{eq:1-MLmGWsSchroEqu2019}, both the source and potential are random, while in \cite{llm2018random}, the potential is assumed to be deterministic.
These two facts make 
this work
%
much more challenging than that in  \cite{llm2018random}. 
The techniques used in the estimates of higher order terms (see Section \ref{sec:AsympHighOrder-MLmGWsSchroEqu2019}) are pseudodifferential operators and microlocal analysis, which are more technically involved compared to that in \cite{llm2018random}.


The rest of this paper is organized as follows. 
In Section \ref{sec:dp-MLmGWsSchroEqu2019}, we first give an introduction to the random model and present some preliminary and auxiliary results. Then we show the well-posedness of the direct scattering problem.
Section \ref{sec:AsympHighOrder-MLmGWsSchroEqu2019} establishes the asymptotics of different terms appeared in the recovery formula.
In Section \ref{sec:recSource-MLmGWsSchroEqu2019}, we recover the rough strength of the source.
Section \ref{sec:recPotential-MLmGWsSchroEqu2019} is devoted to the recovery of the rough strength of the potential. 

\section{Mathematical analysis of the direct problem} \label{sec:dp-MLmGWsSchroEqu2019}

In this section, we show that the direct problem is well-posed in a proper sense.
Before that, we first 
present some preliminaries for the subsequent use
and
give the introduction to our random model.

\subsection{Preliminary and auxiliary results} \label{subsec:preli-SchroEqu2018}

For convenient reference and self-containedness, we first present some preliminary and auxiliary results in what follows. \sq{
In this paper, we mainly focus on the three-dimensional case. Nevertheless, some of the results derived also hold for higher dimensions and in those cases, we choose to present the results in the general dimension $n\geq 3$ since they might be useful in other studies.}

The Fourier transform and inverse Fourier transform of a function $\varphi$ are respectively defined as
\begin{align*} 
& \calF \varphi (\xi) = \widehat \varphi(\xi) := (2\pi)^{-n/2} \int e^{-{\textrm{i}}x\cdot \xi} \varphi(x) \dif x, \\
& \calF^{-1} \varphi (\xi) := (2\pi)^{-n/2} \int e^{{\textrm{i}}x\cdot \xi} \varphi(x) \dif x.
\end{align*}
Set
\[
\Phi(x,y) = \Phi_k(x,y) := \frac {e^{{\textrm{i}}k|x-y|}}{4\pi|x-y|}, \quad x\in\mathbb{R}^3\backslash\{y\}.
\]
$\Phi_k$ is the outgoing fundamental solution, centered at $y$, to the differential operator $-\Delta-k^2$. Define the resolvent operator $\Rk$,
\begin{equation} \label{eq:DefnRk-MLmGWsSchroEqu2019}
(\Rk \varphi)(x) := \int_{\R^3} \Phi_k(x,y) \varphi(y) \dif{y}, \quad x \in \R^3,
\end{equation}
where $\varphi$ can be any measurable function on $\mathbb{R}^3$ as long as \eqref{eq:DefnRk-MLmGWsSchroEqu2019} is well-defined for almost all $x$ in $\R^3$. 

Write $\agl[x] := (1+|x|^2)^{1/2}$ for $x \in \Rn$, $n\geq 1$. We introduce the following weighted $L^p$-norm and the corresponding function space over $\Rn$ for any $\delta \in \R$,
\begin{equation} \label{eq:WetdSpace-MLmGWsSchroEqu2019}
\begin{aligned}
\nrm[L_\delta^p(\Rn)]{\varphi} := \ & \nrm[L^p(\Rn)]{\agl[\cdot]^{\delta} \varphi(\cdot)} = \big( \int_{\Rn} \agl[x]^{p\delta} |\varphi|^p \dif{x} \big)^{\frac 1 p}, \\
L_\delta^p(\Rn) := \ & \{\, \varphi \in L_{loc}^1(\Rn) \,;\, \nrm[L_\delta^p(\Rn)]{\varphi} < +\infty \,\}.
\end{aligned}
\end{equation}
We also define $L_\delta^p(S)$ for any subset $S$ in $\Rn$ by replacing $\Rn$ in \eqref{eq:WetdSpace-MLmGWsSchroEqu2019} with $S$. 
In what follows, we may write $L_\delta^2(\R^3)$ as $L_\delta^2$ for short without ambiguities.
\sq{Let $I$ be the identity operator and} define 
\begin{equation*}
	\sq{\nrm[H_\delta^{s,p}(\Rn)]{f} := \nrm[L_\delta^p(\Rn)]{(I-\Delta)^{s/2} f}, \ H_\delta^{s,p}(\Rn) = \{ f \in \scrS'(\Rn); \nrm[H_\delta^{s,p}(\Rn)]{f} < +\infty\},}
\end{equation*}
where $\scrS'(\Rn)$ stands for the dual space of the Schwartz space $\scrS(\Rn)$.
The space $H_\delta^{s,2}(\Rn)$ is abbreviated as $H_\delta^s(\Rn)$, and $H_0^{s,p}(\Rn)$ is abbreviated as $H^{s,p}(\Rn)$.
It can be verified that
\begin{equation} \label{eq:normEquiv-MLmGWsSchroEqu2019}
\nrm[H_\delta^s(\Rn)]{f} = \nrm[H^\delta(\Rn)]{\agl[\cdot]^s \widehat f(\cdot)}.
\end{equation}

Let $m \in (-\infty,+\infty)$. We define $S^m$ to be the set of all functions $\sigma(x,\xi) \in C^{\infty}(\Rn,\Rn;\mathbb C)$ such that for any two multi-indices $\alpha$ and $\beta$, there is a positive constant $C_{\alpha, \beta}$, depending on $\alpha$ and $\beta$ only, for which
\[
\big| (D_{x}^{\alpha}D_{\xi}^{\beta}\sigma)(x,\xi) \big| \leq C_{\alpha, \beta}(1+|\xi|)^{m-|\beta|}, \quad \forall x, \xi \in \Rn.
\]
We call any function $\sigma$ in $\bigcup_{m \in \R} S^m$ a \emph{symbol}.
A \emph{principal symbol} of $\sigma$ is an equivalent class \sq{$[\sigma] = \{ \tilde \sigma \in S^m \,;\, \sigma - \tilde \sigma \in S^{m-1} \}$}. 
In what follows, we may use one representative $\tilde \sigma$ in $[\sigma]$ to represent the equivalent class $[\sigma]$.
Let $\sigma$ be a symbol. Then the \emph{pseudo-differential operator} $T$, defined on $\scrS(\Rn)$ and associated with $\sigma$, is defined by
\begin{align*}
	(T_{\sigma}\varphi)(x)
	& := (2\pi)^{-n/2} \int_{\Rn} e^{{\textrm{i}}x \cdot \xi} \sigma(x,\xi) \hat{\varphi}(\xi) \dif{\xi} \\
	& \ = (2\pi)^{-n} \iint_{\Rn \times \Rn} e^{{\textrm{i}}(x-y) \cdot \xi} \sigma(x,\xi) \varphi(y) \dif y \dif{\xi}, \quad \forall \varphi \in \scrS(\Rn).
\end{align*}


In the sequel, we write $\mathcal{L}(\mathcal A, \mathcal B)$ to denote the set of all the bounded linear mappings from a normed vector space $\mathcal A$ to a normed vector space $\mathcal B$. 
For any mapping $\mathcal K \in \mathcal{L}(\mathcal A, \mathcal B)$, we denote its operator norm as $\nrm[\mathcal{L}(\mathcal A, \mathcal B)]{\mathcal K}$. 
We also use $C$ and its variants, such as $C_D$, $C_{D,f}$, to denote some generic constants whose particular values may change line by line. 
\sq{For two quantities, we write $\mathcal{P}\lesssim \mathcal{Q}$ to signify $\mathcal{P}\leq C \mathcal{Q}$ and $\mathcal{P} \simeq \mathcal{Q}$ to signify $\widetilde{C}\mathcal{Q}\leq \mathcal{P} \leq C \mathcal{Q}$, for some generic positive constants $C$ and $\widetilde{C}$. }
We write ``almost everywhere'' as~``a.e.''~and ``almost surely'' as~``a.s.''~for short. 
We use $|\mathcal S|$ to denote the Lebesgue measure of any Lebesgue-measurable set $\mathcal S$.

\subsection{The random model} \label{subsec:RandomModel-MLmGWsSchroEqu2019}

As already mentioned in Section \ref{subsec:formu-MLmGWsSchroEqu2019}, a generalized Gaussian random field maps test functions to random variables. Assume $h$ is a generalized Gaussian random field. Then both $\agl[h(\cdot,\omega),\varphi]$ and $\agl[h(\cdot,\omega),\psi]$ are random variables for $\varphi$, $\psi \in \scrS(\Rn)$.
From a statistical point of view, the covariance between these two random variables,
\begin{equation} \label{eq:CovDef-MLmGWsSchroEqu2019}
\mathbb E_\omega ( \agl[\overline{h(\cdot,\omega)},\varphi] \agl[h(\cdot,\omega),\psi]),
\end{equation}
can be understood as the covariance of $h$, where the $\mathbb E_\omega$ means to take expectation on the argument $\omega$. 
Formula \eqref{eq:CovDef-MLmGWsSchroEqu2019} induces an operator $\mathfrak C_h$,
\[
\mathfrak C_h \colon \varphi \in \scrS(\Rn) \ \mapsto \ \mathfrak C_h \varphi \in \scrS'(\Rn),
\]
in a way that
\[
\mathfrak C_h \varphi \colon \psi \in \scrS(\Rn) \ \mapsto \ (\mathfrak C_h \varphi)(\psi) = \mathbb E_\omega ( \agl[\overline{h(\cdot,\omega)},\varphi] \agl[h(\cdot,\omega),\psi] ) \in \mathbb C.
\]
The operator $\mathfrak C_h$ is called the covariance operator of $h$. See also \cite{caro2016inverse,Lassas2008} for reference.

We adopt the definition of the m.i.g.r.~distribution from \cite{caro2016inverse} with some modifications to fit our mathematical setting.

\begin{defn} \label{defn:migr-MLmGWsSchroEqu2019}
	A generalized Gaussian random function $h$ on $\Rn$ is called microlocally isotropic (m.i.g.r.)~with rough order $-m$ and rough strength $\mu(x)$ in $D$, if the following conditions hold:
	\begin{enumerate}
		\item the expectation $\mathbb E h$ is in $C_c^\infty(\Rn)$ with $\supp \mathbb E h \subset D$;
		
		\item $h$ is supported in $D$ a.s.;
		
		\item the covariance operator $\mathfrak C_h$ is a pseudodifferential operator of order \sq{$-m$};
		
		\item $\mathfrak C_h$, regarded as a pseudo-differential operator, has a principal symbol of the form $\mu(x)|\xi|^{-m}$ with $\mu \in C_c^\infty(\Rn;\R)$, $\supp \mu \subset D$ and $\mu(x) \geq 0$ for all $x \in \Rn$.
	\end{enumerate}
\end{defn}

\medskip

Here, \sq{$\mu(x)|\xi|^{-m}$} is a representative of the principal symbol of $\mathfrak C_h$.
Throughout this work, the principal symbol of the covariance operator of the $f(\cdot,\omega)$ in \eqref{eq:1-MLmGWsSchroEqu2019} is assumed to be $\mu_f(x) |\xi|^{-m_f}$ and that of the $q(\cdot,\omega)$ in \eqref{eq:1-MLmGWsSchroEqu2019} is denoted as $\mu_q(x)|\xi|^{-m_q}$.

\begin{lem} \label{lem:migrRegu-MLmGWsSchroEqu2019}
	Let $h$ be a m.i.g.r.~distribution of rough order $-m$ in $D$. Then, $h \in H^{\sq{-s},p}(\Rn)$ almost surely for any $1 < \sq{p} < +\infty$ and \sq{$s > (n-m)/2$}.
\end{lem}
\begin{proof}[Proof of Lemma \ref{lem:migrRegu-MLmGWsSchroEqu2019}]
	See Proposition 2.4 in \cite{caro2016inverse}.
\end{proof}

Lemma \ref{lem:migrRegu-MLmGWsSchroEqu2019} shows the regularity of $h$ according to its rough order.

\smallskip

By the Schwartz kernel theorem (see Theorem 5.2.1 in \cite{hormander1985analysisI}), there exists a kernel $K_h(x,y)$ with $\supp K_h \subset D \times D$ such that
\begin{equation} \label{eq:CK-MLmGWsSchroEqu2019}
(\mathfrak C_h \varphi)(\psi) 
= \mathbb E_\omega ( \agl[\overline{h(\cdot,\omega)},\varphi] \agl[h(\cdot,\omega),\psi]) 
= \iint K_h(x,y) \varphi(x) \psi(y) \dif x \dif y,
\end{equation}
for all $\varphi$, $\psi \in \scrS(\Rn)$.
It is easy to verify that $K_h(x,y) = \overline{K_h(y,x)}$.
Denote the symbol of $\mathfrak C_h$ as $c_h$, then it can be verified \cite{caro2016inverse} that the equalities
\begin{subequations} \label{eq:KandSymbol-MLmGWsSchroEqu2019}
	\begin{numcases}{}
	K_h(x,y) = (2\pi)^{-n} \int e^{{\textrm{i}}(x-y) \cdot \xi} c_h(x,\xi) \dif \xi, \label{eq:KtoSymbol-MLmGWsSchroEqu2019} \\
	c_h(x,\xi) = \int e^{-{\textrm{i}}\xi\cdot(x-y)} K_h(x,y) \dif x, \label{eq:SymboltoK-MLmGWsSchroEqu2019}
	\end{numcases}
\end{subequations}
hold in the distributional sense, and the integrals in \eqref{eq:KandSymbol-MLmGWsSchroEqu2019} shall be understood as oscillatory integrals.
{Despite} the fact that $h$ usually is not a function, {intuitively speaking, however,} it is helpful to keep in mind the following correspondence,
\[
K_h(x,y) \sim \mathbb E_\omega \big( \overline{h(x,\omega)} h(y,\omega) \big).
\]

\subsection{The well-posedness of the direct problem} \label{subsec:DP-MLmGWsSchroEqu2019}

We first derive two important quantitative estimates. 

\begin{thm} \label{thm:RkBounded-MLmGWsSchroEqu2019}
	For any $0 < s < 1/2$ and $\epsilon > 0$, when $k > 2$,
	$$\nrm[H_{-1/ 2 - \epsilon}^s(\R^3)]{\Rk \varphi} \leq C_{\epsilon,s} k^{-(1 - 2s)} \nrm[H_{1/2 + \epsilon}^{-s}(\R^3)]{\varphi}, \quad \varphi \in H_{1/2 + \epsilon}^{-s}(\R^3).$$
\end{thm}

\begin{thm} \label{thm:VBounded-MLmGWsSchroEqu2019}
	Assume that $q(\cdot,\omega)$ is microlocally isotropic of order $-m$. 
	Then 
	\sq{in any dimension $n \geq 3$ and for} 
	every $s > (n-m)/2$ and every $\epsilon \in (0, 3/2]$, $q \colon H_{-1/2 - \epsilon}^s(\Rn) \to H_{1/2 + \epsilon}^{-s}(\Rn)$ is bounded almost surely,
	\[
	\nrm[H_{1/2 + \epsilon}^{-s}(\R^3)]{q(\cdot,\omega) \varphi(\cdot)} \leq C_{\epsilon,s}(\omega) \nrm[H_{-1/2 - \epsilon}^s(\Rn)]{\varphi}, \quad \varphi \in H_{1/2 + \epsilon}^{-s}(\Rn), \quad \aee \omega \in \Omega.
	\]
	The random variable $C_{\epsilon,s}(\omega)$ is finite almost surely.
\end{thm}

The arguments in proving Theorems \ref{thm:RkBounded-MLmGWsSchroEqu2019} and \ref{thm:VBounded-MLmGWsSchroEqu2019} are inspired by \cite{caro2016inverse} and
[$\S 29$,\,\citen{eskin2011lectures}].

\begin{proof}[Proof of Theorem \ref{thm:RkBounded-MLmGWsSchroEqu2019}]
	\sq{Define an operator }
	\begin{equation}\label{eq:opnnna1}
	\sq{\mathcal{R}_{k,\tau} \varphi(x) := (2\pi)^{-3/2} \int_{\R^3} e^{{\textrm{i}}x\cdot \xi} \frac {\hat \varphi(\xi)} {|\xi|^2 - k^2 - \textrm{i}\tau} \dif \xi,}
	\end{equation}
	\sq{where $\tau\in\mathbb{R}_+$. }
	Fix a function $\chi$ satisfying
	\begin{equation} \label{eq:cutoffFunc-MLmGWsSchroEqu2019}
	\left\{\begin{aligned} 
	& \chi \in C_c^\infty(\Rn),\, 0 \leq \sq{\chi} \leq 1, \\
	& \chi(x) = 1 \mbox{ when } |x| \leq 1, \\
	& \chi(x) = 0 \mbox{ when } |x| \geq 2.
	\end{aligned}\right.
	\end{equation}
	Write $\mathfrak{R} \psi(x) := \psi(-x)$. Fix $p \in (1,+\infty)$, we have
	\begin{align}
	& \ (\mathcal{R}_{k,\tau} \varphi,\psi)_{L^2(\R^3)} \nonumber\\
	= \ & \ \int_{\R^3} \mathcal{R}_{k,\tau} \varphi (x) \overline{\psi(x)} \dif x = \int_{\R^3} \calF\{\mathcal{R}_{k,\tau} \varphi\} (\xi) \cdot \calF\{\mathfrak{R} \overline \psi\}(\xi) \dif \xi \nonumber\\
	= \ & \ \int_0^\infty \frac {(1 - \chi^2(r - k))} {r^2 - k^2 - {\textrm{i}}\tau} \dif r \cdot \int_{|\xi| = r} \hat \varphi (\xi) \cdot \widehat{\mathfrak{R} \overline \psi}(\xi) \dif{S(\xi)} \nonumber\\
	& \ + \int_0^\infty \frac {\agl[r]^{1/p}\, r^2 \chi^2(r - k)} {r^2 - k^2 - {\textrm{i}}\tau} \dif r \cdot \int_{\mathbb{S}^2} [\agl[k]^{\frac {-1} {2p}} \hat \varphi (k\omega)] [\agl[k]^{\frac {-1} {2p}} \widehat{\mathfrak{R} \overline \psi}(k\omega)] \dif{S(\omega)} \nonumber\\
	& \ + \int_0^\infty \frac {\agl[r]^{1/p}\, r^2 \chi^2(r - k)} {r^2 - k^2 - {\textrm{i}}\tau} \dif r \cdot \int_{\mathbb{S}^2} \{ [\agl[r]^{\frac {-1} {2p}} \hat \varphi (r\omega)] [\agl[r]^{\frac {-1} {2p}} \widehat{\mathfrak{R} \overline \psi}(r\omega)] \nonumber\\
	& \hspace*{33ex} - [\agl[k]^{\frac {-1} {2p}} \hat \varphi (k\omega)] [\agl[k]^{\frac {-1} {2p}} \widehat{\mathfrak{R} \overline \psi}(k\omega)] \} \dif{S(\sq{\omega})} \nonumber\\
	=: & \ I_1(\tau) + I_2(\tau) + I_3(\tau). \label{eq:RkBoundedI123-MLmGWsSchroEqu2019}
	\end{align}

	\smallskip

	Now we estimate $I_1(\tau)$. By Young's inequality we have
	\begin{equation} \label{eq:YoungIneq-MLmGWsSchroEqu2019}
	ab \leq a^p/p + b^q/q \quad \Rightarrow \quad (p^{1/p} q^{1/q}) a^{1/p} b^{1/q} \leq a + b
	\end{equation}
	for $a,b > 0,\, p,q > 1,\, 1/p + 1/q = 1$. 
	Note that $|\widehat{\mathfrak{R} \overline \psi}(\xi)| = |\hat \psi(\xi)|$, one can compute
	\begin{align}
	|I_1(\tau)|
	& \leq \int_0^\infty \frac {1 - \chi^2(r - k)} {1 \cdot |r - k|} \dif r \cdot \int_{|\xi| = r} |\hat \varphi (\xi)| \cdot |\widehat{\mathfrak{R} \overline \psi}(\xi)| \dif{S(\xi)} \nonumber\\
	& \leq \int_0^\infty \frac {1 - \chi^2(r - k)} {1 \cdot p^{1/p} q^{1/q} (r+1)^{1/p} (k-1)^{1/q}} \dif r \cdot \int_{|\xi| = r} |\hat \varphi (\xi)| \cdot |\hat \psi(\xi)| \dif{S(\xi)} \quad (\text{by } \eqref{eq:YoungIneq-MLmGWsSchroEqu2019}) \nonumber\\
	& \leq C_p k^{-1/q} \int_0^\infty \agl[r]^{-1/p} \dif r \cdot \int_{|\xi| = r} |\hat \varphi (\xi)| \cdot |\hat \psi(\xi)| \dif{S(\xi)} \nonumber\\
	& \leq C_p k^{1/p - 1} \nrm[H_\delta^{-1/(2p)}(\R^3)]{\varphi} \nrm[H_\delta^{-1/(2p)}(\R^3)]{\psi}, \label{eq:RkBoundedI1-MLmGWsSchroEqu2019}
	\end{align}
	where $1 < p < +\infty$ and $\delta > 0$ and the $C_p$ is independent of $\tau$.

	\smallskip

	We next estimate $I_2(\tau)$. One has
	\begin{align}
	I_2(\tau)
	& = \int_{\mathbb{S}^2} [\agl[k]^{\frac {-1} {2p}} \hat \varphi (k\omega)] [\agl[k]^{\frac {-1} {2p}} \widehat{\mathfrak{R} \overline \psi}(k\omega)] \int_0^\infty \frac {\agl[r]^{\frac 1 p} r^2 \chi^2(r - k) \dif r} {r^2 - k^2 - {\textrm{i}}\tau} \dif{S(\omega)}. \label{eq:RkBoundedI2Inter1-MLmGWsSchroEqu2019}
	\end{align}
	Let $\tau_0 \in (0,1)$ be a fixed number whose value shall be specified later. 
	Write $p_\tau(r) := p(r) = r^2 - k^2 - {\textrm{i}}\tau$. Recall that $\chi(r-k) = 0$ when $|r - k| > 2$. When $\tau_0 \leq |r - k| \leq 2$, we have
	\begin{equation} \label{eq:p1-MLmGWsSchroEqu2019}
	|p(r)| \geq |\Re p(r)| = |r - k| |r + k| \geq \tau_0 (2k - 2) \geq \tau_0 k.
	\end{equation}
	Write $\Gamma_{k,\tau_0} := \{r \in \mathbb C ; |r - k| = \tau_0, \Im r \leq 0 \}$. When $r \in \Gamma_{k,\tau_0}$, we have
	\begin{equation} \label{eq:p2-MLmGWsSchroEqu2019}
		\forall \tau \in (0,\tau_0) , \quad |p_\tau(r)| \geq |r-k| \,|2k + (r-k)| - \tau_0 = \tau_0 (2k - \tau_0) - \tau_0 \geq \tau_0 k.
	\end{equation}
	Combining \eqref{eq:p1-MLmGWsSchroEqu2019} and \eqref{eq:p2-MLmGWsSchroEqu2019}, we conclude that $\forall \tau \in (0,\tau_0), \forall k > 2$,
	\begin{equation} \label{eq:p-MLmGWsSchroEqu2019}
	|p_\tau(r)| \geq \tau_0 k, \quad \Forall r \in \{ r \in \R_+; 2 \geq |r - k| \geq \tau_0 \} \cup \Gamma_{k,\tau_0},\, \forall \tau \in (0,\tau_0).
	\end{equation}
	By using Cauchy's integral theorem, we change the integral domain w.r.t.~$r$ in \eqref{eq:RkBoundedI2Inter1-MLmGWsSchroEqu2019} from $\R_+$ to $\{ r \in \R_+; 2 \geq |r - k| \geq \tau_0 \} \cup \Gamma_{k,\tau_0}$. Combining this with the estimate \eqref{eq:p-MLmGWsSchroEqu2019} and noting that $\chi(r - k) = 1$ when $r \in \{ r \in \R; |r-k| \leq 1 \}$, we have
	\begin{align}
	|I_2(\tau)|
	& \leq \int_{|\xi|=k} \agl[\xi]^{\frac {-1} {2p}} |\hat \varphi (\xi)| \cdot \agl[\xi]^{\frac {-1} {2p}} |\hat \psi(\xi)| \big( \int_{\{ r \in \R_+ \,;\, 2 \geq |r - k| \geq \tau_0 \}} \frac {\agl[r]^{\frac 1 p} (r/k)^2} {\tau_0 k} \dif r \big) \dif{S(\xi)} \nonumber\\
	& \ \ \ + \int_{|\xi|=k} \agl[\xi]^{\frac {-1} {2p}} |\hat \varphi (\xi)| \cdot \agl[\xi]^{\frac {-1} {2p}} |\hat \psi(\xi)| \big( \int_{\Gamma_{k,\tau_0}} \frac {(1+|r|^2)^{\frac 1 {2p}} (|r|/k)^2} {\tau_0 k} \dif r \big) \dif{S(\xi)} \label{eq:I2Inter1-MLmGWsSchroEqu2019}
	\end{align}
	for all $\tau \in (0,\tau_0)$ and for all $ k > 2$.
	
	Note that in $\{ r \in \R_+; 2 \geq |r - k| \geq \tau_0 \}$ we have
	\begin{equation} \label{eq:T2r1-MLmGWsSchroEqu2019}
	\agl[r]^{2s} \leq 5^s \agl[k]^{2s}, \quad 1 \leq (r/k)^2 \leq 4.
	\end{equation}
	For $r \in \Gamma_{k,\tau_0}$ the complex number $(1+r^2)$ can be expressed as $R(r) e^{{\textrm{i}}\theta(r)}$ for real valued functions $R(r)$ and $\theta(r)$. 
	Now we choose $\tau_0$ small enough such that $|\theta(r)| < \frac \pi {10}$ in $\Gamma_{2,\tau_0}$, then $|\theta(r)| < \frac \pi {10}$ in $\Gamma_{k,\tau_0}$ for all $k \geq 2$. This can be easily seen from the geometric view. 
	Thus $(1+r^2)^s$ is well-defined for all $|s| \leq 2$, and
	\begin{equation} \label{eq:T2r2-MLmGWsSchroEqu2019}
	\forall r \in \Gamma_{k,\tau_0}, \quad |(1+r^2)^s| = |1+r^2|^s \leq (1+|r|^2)^s \leq \agl[k+\tau_0]^{2s} \leq C_{\tau_0} \agl[k]^{2s}
	\end{equation}
	for some constant $C_{\tau_0}$ independent of $\tau$ when $0 < \tau < \tau_0$. 
	Similarly, we have
	\begin{equation} \label{eq:T2r3-MLmGWsSchroEqu2019}
	\forall r \in \Gamma_{k,\tau_0}, \quad |r/k|^2 \leq (k+\tau_0)^2/k^2 \leq C_{\tau_0}
	\end{equation}
	for some constant $C_{\tau_0}$ independent of $\tau$.
	Hence by \eqref{eq:T2r1-MLmGWsSchroEqu2019}, \eqref{eq:T2r3-MLmGWsSchroEqu2019} and Remark 13.1 in \cite{eskin2011lectures}, we can continue \eqref{eq:I2Inter1-MLmGWsSchroEqu2019} as
	\begin{align}
	|I_2(\tau)|
	& \leq C_{\tau_0} \int_{|\xi|=k} \agl[\xi]^{\frac {-1} {2p}} |\hat \varphi (\xi)| \agl[\xi]^{\frac {-1} {2p}} |\hat \psi(\xi)| \big( \int_{\Gamma_{k,\tau_0} \cup \{ r \in \R_+; 2 \geq |r - k| \geq \tau_0 \}} \frac {\agl[k]^{1/p}} {\tau_0 k} \dif r \big) \dif{S(\xi)} \nonumber\\
	& \leq C_{\tau_0} k^{1/p - 1} \nrm[H^{1/2 + \epsilon}(\R^3)]{\agl[\cdot]^{-1/(2p)} \hat \varphi(\cdot)} \nrm[H^{1/2 + \epsilon}(\R^3)]{\agl[\cdot]^{-1/(2p)} \hat \psi(\cdot)} \nonumber\\
	& \leq C_{\tau_0,\epsilon} k^{1/p - 1} \nrm[H_{1/2 + \epsilon}^{-1/(2p)}(\R^3)]{\varphi} \nrm[H_{1/2 + \epsilon}^{-1/(2p)}(\R^3)]{\psi}, \label{eq:RkBoundedI2-MLmGWsSchroEqu2019}
	\end{align}
	where the constant $C_{\tau_0,\epsilon}$ is independent of $\tau$.
	\sq{It should be pointed out that the presence of the infinitesimal number $\epsilon$ in $\nrm[H^{1/2 + \epsilon}]{\cdot}$ in \eqref{eq:RkBoundedI2-MLmGWsSchroEqu2019} comes from the requirement that the order of the Sobolev space should be strictly greater than $1/2$; see Remark 13.1 in \cite{eskin2011lectures} for more relevant discussion.}
	{\color{blue} Here, in deriving the last inequality in \eqref{eq:RkBoundedI2-MLmGWsSchroEqu2019}, we have made use of \eqref{eq:normEquiv-MLmGWsSchroEqu2019}.}

	\smallskip

	Finally, we estimate $I_3(\tau)$. Denote
	$\mathbb F(r\omega) = \mathbb F_r(\omega) := \agl[r]^{-1/(2p)} \hat \varphi(r\omega)$ and $\mathbb G(r\omega) = \mathbb G_r(\omega) := \agl[r]^{-1/(2p)} \widehat{\mathfrak R \bar \psi}(r\omega)$. One can compute
	\begin{align}
	|I_3(\tau)|
	& = \big| \int_0^\infty \frac {\agl[r]^{1/p}\, r^2 \chi^2(r - k)} {r^2 - k^2 - {\textrm{i}}\tau} \dif r \cdot \int_{\mathbb{S}^2} (\mathbb F_r \mathbb G_r - \mathbb F_k \mathbb G_k) \dif{S(\omega)} \big| \nonumber\\
	& \leq \int_0^\infty \frac {\agl[r]^{1/p} \chi^2(r - k)} {|r^2 - k^2|} \cdot \nrm[L^2(\mathbb S_r^2)]{\mathbb F_r} \cdot \big( r^2 \int_{\mathbb{S}^2} |\mathbb G_r - \mathbb G_k|^2 \dif{S(\omega)} \big)^{\frac 1 2} \dif r \nonumber\\
	& \ \ \ + \int_0^\infty \frac {\agl[r]^{1/p} \chi^2(r - k)} {|r^2 - k^2|} \cdot \big( r^2 \int_{\mathbb{S}^2} |\mathbb F_r - \mathbb F_k|^2 \dif{S(\omega)} \big)^{\frac 1 2} \cdot \big(\frac r k \big)^2 \nrm[L^2(\mathbb S_k^2)]{\mathbb G_k} \dif r, \label{eq:RkBoundedI3.1-MLmGWsSchroEqu2019}
	\end{align}
	where $\mathbb S_r^2$ signifies the central sphere of radius $r$.
	Combining both Remark 13.1 and (13.28) in \cite{eskin2011lectures} and \eqref{eq:normEquiv-MLmGWsSchroEqu2019} and \eqref{eq:YoungIneq-MLmGWsSchroEqu2019}, we can continue \eqref{eq:RkBoundedI3.1-MLmGWsSchroEqu2019} as
	\begin{align}
	|I_3(\tau)|
	& \leq C_{\alpha,\epsilon} \int_0^\infty \frac {\agl[r]^{1/p} \chi^2(r - k)} {|r - k| (r+k)} \cdot \nrm[H^{1/2+\epsilon}(\R^3)]{\mathbb F} \cdot |r - k|^\alpha \cdot \nrm[H^{1/2+\epsilon}(\R^3)]{\mathbb G} \dif r \nonumber\\
	& \leq C_{\alpha,\epsilon,p} \int_0^\infty \frac {\agl[r]^{1/p} \chi^2(r - k)} {|r - k|^{1-\alpha} (r+1)^{1/p} (k-1)^{1-1/p}} \dif r \cdot \nrm[H^{1/2+\epsilon}(\R^3)]{\mathbb F} \nrm[H^{1/2+\epsilon}(\R^3)]{\mathbb G} \nonumber\\
	& \leq C_{\alpha,\epsilon,p} k^{1/p-1} \int_0^\infty \frac {\chi^2(r - k)} {|r - k|^{1-\alpha}} \dif r \cdot \nrm[H_{1/2+\epsilon}^{-1/(2p)}(\R^3)]{\varphi} \cdot \nrm[H_{1/2+\epsilon}^{-1/(2p)}(\R^3)]{\psi} \nonumber\\
	& \leq C_{\alpha,\epsilon,p} k^{1/p-1} \nrm[H_{1/2+\epsilon}^{-1/(2p)}(\R^3)]{\varphi} \cdot \nrm[H_{1/2+\epsilon}^{-1/(2p)}(\R^3)]{\psi}, \label{eq:RkBoundedI3-MLmGWsSchroEqu2019}
	\end{align}
	where the $\epsilon$ can be any positive real number and  the $\alpha$ satisfies $0 < \alpha < \epsilon$, and the constant $C_{\alpha,\epsilon,p}$ is independent of $\tau$.
	
	Combining \eqref{eq:RkBoundedI123-MLmGWsSchroEqu2019}, \eqref{eq:RkBoundedI1-MLmGWsSchroEqu2019}, \eqref{eq:RkBoundedI2-MLmGWsSchroEqu2019} and \eqref{eq:RkBoundedI3-MLmGWsSchroEqu2019}, we arrive at
	\begin{equation*}
	|(\mathcal{R}_{k,\tau} \varphi,\psi)_{L^2(\R^3)}| \leq |I_1(\tau)| + |I_2(\tau)| + |I_3(\tau)| 
	\leq C k^{1/p - 1} \nrm[H_{1/2 + \epsilon}^{-1/(2p)}(\R^3)]{\varphi} \nrm[H_{1/2 + \epsilon}^{-1/(2p)}(\R^3)]{\psi},
	\end{equation*}
	which implies that
	\begin{equation} \label{eq:RkeBoundedf-MLmGWsSchroEqu2019}
	\nrm[H_{-1/2 - \epsilon}^{1/(2p)}(\R^3)]{\mathcal{R}_{k,\tau} \varphi} 
	\leq C k^{1/p - 1} \nrm[H_{1/2 + \epsilon}^{-1/(2p)}(\R^3)]{\varphi}
	\end{equation}
	for some constant $C$ independent of $\tau$.

	\medskip

	Next we study the limiting case $\lim\limits_{\tau \to 0^+} \mathcal{R}_{k,\tau} \varphi$. For any two positive real numbers $\tau_1, \tau_2 < \tilde \tau$, we study $|I_j(\tau_1) - I_j(\tau_2)|$ for $j = 1,2,3$.
	
	Similar to our previous derivation, we have
	\begin{align}
	|I_1(\tau_1) - I_1(\tau_2)|
	& \leq \int_0^\infty \frac {|\tau_1 - \tau_2| (1 - \chi^2(r - k))} {|r^2 - k^2| \cdot p^{\frac 1 p} q^{\frac 1 q} (r+1)^{\frac 1 p} (k-1)^{\frac 1 q}} \dif r \cdot \int_{|\xi| = r} |\hat \varphi (\xi)| \cdot |\hat \psi(\xi)| \dif{S(\xi)} \nonumber\\
	& \leq \tilde \tau \, C_p k^{1/p - 1} \nrm[H_\delta^{-1/(2p)}(\R^3)]{\varphi} \nrm[H_\delta^{-1/(2p)}(\R^3)]{\psi}, \label{eq:Rk12BoundedI1-MLmGWsSchroEqu2019}
	\end{align}
	and
	\begin{align}
	|I_2(\tau_1) - I_2(\tau_2)|
	& \leq C \int_{|\xi|=k} \agl[\xi]^{\frac {-1} {2p}} |\hat \varphi (\xi)| \agl[\xi]^{\frac {-1} {2p}} |\hat \psi(\xi)| \big( \int_{\{ r \in \R_+; 2 \geq |r - k| \geq \tau_0 \}} \frac {|\tau_1 - \tau_2| \agl[k]^{\frac 1 p}} {(\tau_0 k)^2} \dif r \big) \dif{S(\xi)} \nonumber\\
	& \ \ \ + C \int_{|\xi|=k} \agl[\xi]^{\frac {-1} {2p}} |\hat \varphi (\xi)| \agl[\xi]^{\frac {-1} {2p}} |\hat \psi(\xi)| \big( \int_{\Gamma_{k,\tau_0}} \frac {|\tau_1 - \tau_2| \agl[k]^{\frac 1 p}} {(\tau_0 k)^2} \dif r \big) \dif{S(\xi)} \nonumber\\
	& \leq \tilde \tau \, C k^{1/p - 1} \nrm[H_{1/2 + \epsilon}^{-1/(2p)}(\R^3)]{\varphi} \nrm[H_{1/2 + \epsilon}^{-1/(2p)}(\R^3)]{\psi}. \label{eq:Rk12BoundedI2-MLmGWsSchroEqu2019}
	\end{align}
	
	To analyze $I_3(\tau)$ as $\tau$ goes to zero, we note that by \eqref{eq:YoungIneq-MLmGWsSchroEqu2019} one has
	\begin{equation*}
	C_\beta (\Re z)^\beta (\Im z)^{1-\beta} \leq |z|, \quad \forall z \in \mathbb{C},
	\end{equation*}
	which holds for all $\beta \in (0,1)$ and some constant $C_\beta$. Without loss of generality, we assume $\tau_1 \leq \tau_2$. Hence we can compute
	\begin{align*}
	\big| \frac 1 {r^2 - k^2 - {\textrm{i}}\tau_1} - \frac 1 {r^2 - k^2 - {\textrm{i}}\tau_2} \big|
	& \leq \frac 1 {|r^2 - k^2|} \cdot \frac {C \tau_2} {|r^2 - k^2|^\beta \cdot \tau_2^{1-\beta}} 
	\leq \frac {C \tau_2^\beta} {|r^2 - k^2|^{1+\beta}}. 
	\end{align*}
	Thus
	\begin{align}
	& \ |I_3(\tau_1) - I_3(\tau_2)| \nonumber\\
	\lesssim & \ \tau_2^\beta \int_0^\infty \frac {\agl[r]^{\frac 1 p} \chi^2(r - k)} {|r-k|^{1+\beta} (r+1)^{\frac 1 p} (k-1)^{1 - \frac 1 p}} \cdot |r-k|^\alpha \cdot \nrm[H^{\frac 1 2 + \epsilon}(\R^3)]{\mathbb F} \cdot \nrm[H^{\frac 1 2 + \epsilon}(\R^3)]{\mathbb G} \dif r \nonumber\\
	\lesssim & \ \tilde\tau^\beta k^{1/p - 1} \nrm[H_{1/2 + \epsilon}^{-1/(2p)}(\R^3)]{\varphi} \nrm[H_{1/2 + \epsilon}^{-1/(2p)}(\R^3)]{\psi}, \label{eq:Rk12BoundedI3-MLmGWsSchroEqu2019}
	\end{align}
       where the last inequality holds when $0 < \beta < \alpha$.
	
	From \eqref{eq:Rk12BoundedI1-MLmGWsSchroEqu2019}, \eqref{eq:Rk12BoundedI2-MLmGWsSchroEqu2019} and \eqref{eq:Rk12BoundedI3-MLmGWsSchroEqu2019} we arrive at
	\begin{equation} \label{eq:LimAbsPrin1-MLmGWsSchroEqu2019}
	\nrm[H_{-1/2 - \epsilon}^{-1/(2p)}(\R^3)]{\mathcal{R}_{k,\tau_1} \varphi - \mathcal{R}_{k,\tau_2} \varphi} 
	\lesssim \tilde \tau \nrm[H_{1/2 + \epsilon}^{-1/(2p)}(\R^3)]{\varphi}, \quad \forall \tau_1,\tau_2 \in (0,\tilde \tau),
	\end{equation}
	and thus $\mathcal{R}_{k,\tilde \tau} \varphi$ converges and
	\begin{equation} \label{eq:LimAbsPrin2-MLmGWsSchroEqu2019}
	\lim_{\tilde \tau \to 0^+} \mathcal{R}_{k,\tilde \tau} \varphi = \Rk \varphi \quad\text{ in }\quad H_{-1/2 - \epsilon}^{1/(2p)}(\R^3).
	\end{equation}
	The relationships \eqref{eq:LimAbsPrin1-MLmGWsSchroEqu2019} and \eqref{eq:LimAbsPrin2-MLmGWsSchroEqu2019} sometimes refer to as the \emph{limiting absorption principle}. Hence from \eqref{eq:RkeBoundedf-MLmGWsSchroEqu2019} and \eqref{eq:LimAbsPrin2-MLmGWsSchroEqu2019} we conclude that
	$$\nrm[H_{-1/2 - \epsilon}^{1/(2p)}(\R^3)]{\Rk \varphi} \leq C_{\epsilon,p} k^{-(1 - 1/p)} \nrm[H_{1/2 + \epsilon}^{-1/(2p)}(\R^3)]{\varphi}$$
	holds for any $1 < p < +\infty$ and any $\epsilon > 0$. 
	
	The proof is complete.
\end{proof}

\medskip

\begin{proof}[Proof of Theorem \ref{thm:VBounded-MLmGWsSchroEqu2019}]
	Let $\varphi,\psi \in \scrS(\Rn)$ and define $\agl[q \varphi,\psi] := \agl[q,\varphi \psi]$. Choose a function $\chi$ such that
	$\chi \in C_c^\infty(\Rn)$ and $\chi(x) = 1$ when $x \in \supp q$. Choose $s'$ satisfying $-s' < (m-n)/2$ and $p, p'$ satisfying $1 < p < +\infty,\, 1/{p'} + 1/p = 1$. Then according to [Proposition 2.4, \citen{caro2016inverse}], $\nrm[H^{-s',p'}(\R^n)]{q} < +\infty$ almost surely. Denote $\nrm[H^{-s',p'}(\R^n)]{q}$ as $C_s(\omega)$. One can compute
	\begin{align}
	|\agl[q \varphi,\psi]|
	& = |\agl[q,(\chi \varphi) (\chi \psi)]| = |\agl[(I-\Delta)^{-s'} q,(I-\Delta)^{s'} \big( (\chi \varphi) (\chi \psi) \big) ]| \nonumber\\
	& \leq \nrm[H^{-s',p'}(\R^n)]{q} \cdot \nrm[L^p(\R^n)]{(I-\Delta)^{s'} \big( (\chi \varphi) (\chi \psi) \big) } \nonumber\\
	& = C_s(\omega) \nrm[L^p(\R^n)]{(I-\Delta)^{s'} \big( (\chi \varphi) (\chi \psi) \big) }. \label{eq:VBddInter1-MLmGWsSchroEqu2019}
	\end{align}
	According to the fractional Leibniz rule \cite{grafakos2014KatoPonce}, when $1/p = 1/2 + 1/q$, one has
	\begin{align}
	\nrm[L^p(\R^n)]{(I-\Delta)^{s'} \big( (\chi \varphi) (\chi \psi) \big) } 
	& \leq C_s(\omega) \big( \nrm[L^2(\R^n)]{\chi \varphi} \nrm[H^{s',q}(\R^n)]{\chi \psi} \nonumber\\
	& \ \ + \nrm[L^2(\R^n)]{\chi \psi} \nrm[H^{s',q}(\R^n)]{\chi \varphi} \big). \label{eq:FracLeibRule-MLmGWsSchroEqu2019}
 	\end{align}
	By \eqref{eq:VBddInter1-MLmGWsSchroEqu2019}-\eqref{eq:FracLeibRule-MLmGWsSchroEqu2019} and noting the Sobolev embedding $H^s(\R^n) \hookrightarrow H^{s',q}(\R^n)$ when $s - n/2 \geq s' - n/q$, $s > s'$, we can continue \eqref{eq:VBddInter1-MLmGWsSchroEqu2019} as
	\begin{align}
	|\agl[q \varphi,\psi]|
	& \lesssim C_s(\omega) \big( \nrm[L^2(\R^n)]{\chi \varphi} \cdot \nrm[H^{s',q}(\R^n)]{\chi \psi} + \nrm[L^2(\R^n)]{\chi \psi} \cdot \nrm[H^{s',q}(\R^n)]{\chi \varphi} \big) \nonumber\\
	& \lesssim C_s(\omega) \nrm[H^s(\R^n)]{\chi \varphi} \cdot \nrm[H^s(\R^n)]{\chi \psi}. \label{eq:VBddInter2-MLmGWsSchroEqu2019}
	\end{align}
	Because $1 < p' < +\infty$ and $s' > -\frac {m-n} 2$, the real number $s$ should satisfy
	$$s \geq s' + \frac n 2 - \frac n q = s' + \frac n 2 - n(\frac 1 p - \frac 1 2) = s' + n - \frac n p = s' + \frac n {p'} \geq s' >  \frac {n-m} 2.$$
	
	Next we adapt the proof of Lemma 3.7 in \cite{caro2016inverse} to show that
	\begin{equation} \label{eq:fWeighted-MLmGWsSchroEqu2019}
	\nrm[H^s(\R^n)]{\chi \varphi} \leq C \nrm[H_{-2}^s(\R^n)]{\varphi}, \quad \varphi \in \scrS(\Rn).
	\end{equation}
	Rewriting the right-hand side of \eqref{eq:fWeighted-MLmGWsSchroEqu2019} in terms of the $L^2$-norm form, we obtain
	\[
	\nrm[H^s(\R^n)]{\chi \varphi} \leq C \nrm[L^2(\R^n)]{\agl[\cdot]^{-2} (I - \Delta)^{s/2} \varphi}.
	\]
	Write $\psi(x) := \agl[x]^{-2} (I - \Delta)^{s/2} \varphi(x)$. Obviously, $\varphi \in \scrS(\Rn)$ is equivalent to $\psi \in \scrS(\Rn)$. Define $T_a \psi := \chi \cdot (I - \Delta)^{-s/2} (\agl[\cdot]^2 \psi)$. Then $\chi \varphi = T_a \psi$ and \eqref{eq:fWeighted-MLmGWsSchroEqu2019} is equivalent to
	\begin{equation} \label{eq:gWeighted-MLmGWsSchroEqu2019}
	\nrm[H^s(\R^n)]{T_a \psi} \leq C \nrm[L^2(\R^n)]{\psi}.
	\end{equation}
	$T_a$ is a pseudo-differential operator with
	$$a(x,\xi) := \chi(x) \big( \agl[x]^2 \agl[\xi]^{-s} - 2{\textrm{i}}x \cdot \nabla_\xi \agl[\xi]^{-s} - \Delta_\xi \agl[\xi]^{-s} \big)$$
	as its symbol. It is easy to see that $a \in S^{-s}$, and thus according to the properties of pseudo-differential operators \cite{eskin2011lectures}, \eqref{eq:gWeighted-MLmGWsSchroEqu2019} holds, and so does \eqref{eq:fWeighted-MLmGWsSchroEqu2019}.
	
	We can continue the estimates in \eqref{eq:VBddInter2-MLmGWsSchroEqu2019} as
	\begin{align*}
	|\agl[q \varphi,\psi]|
	& \lesssim C_s(\omega) \nrm[H^s(\R^n)]{\chi \varphi} \cdot \nrm[H^s(\R^n)]{\chi \psi} \lesssim C_s(\omega) \nrm[H_{-2}^s(\R^n)]{\varphi} \cdot \nrm[H_{-2}^s(\R^n)]{\psi} \\\
	& \leq C_s(\omega) \nrm[H_{-1/2 - \epsilon}^s(\R^n)]{\varphi} \cdot \nrm[H_{-1/2 - \epsilon}^s(\R^n)]{\psi}, \quad \forall \varphi, \psi \in \scrS(\Rn), 
	\end{align*}
	\sq{where $0 < \epsilon \leq 3/2$}, which implies that
	\begin{equation} \label{eq:VBdd-MLmGWsSchroEqu2019}
	\nrm[H_{1/2 + \epsilon}^{-s}(\R^n)]{q \varphi} \leq C_{\epsilon,s}(\omega)  \nrm[H_{-1/2 - \epsilon}^s(\R^n)]{\varphi}, \quad \forall \varphi \in \scrS(\Rn).
	\end{equation}
	
	We proceed to show that $\scrS(\Rn)$ is dense in $H_{-1/2 - \epsilon}^s(\R^n)$. Fix a function $\varphi$ satisfying \eqref{eq:cutoffFunc-MLmGWsSchroEqu2019}. It is clear that $\varphi \in H_{-1/2 - \epsilon}^s(\R^n)$, and hence we have $\agl[\cdot]^{-1/2 - \epsilon} (I - \Delta)^{s/2} \varphi \in L^2(\Rn)$. Then for any $\delta > 0$ there exists a constant $M$, depending on $\varphi$, such that $\nrm[L^2(\Rn)]{\agl[\cdot]^{-1/2 - \epsilon} (I - \Delta)^{s/2} \varphi - \varphi^{(1)}} < \frac \delta 2$, where $\varphi^{(1)} = \varphi(\cdot/M) \agl[\cdot]^{-1/2 - \epsilon} (I - \Delta)^{s/2} \varphi$. Note that $\varphi^{(1)} \in L^2(\Rn)$ with a compact support. {\color{blue} Furthermore, there exists a sufficiently small constant $\zeta\in\mathbb{R}_+$ such that $\nrm[L^2(\Rn)] {\varphi^{(1)} - \varphi^{(2)}} < \frac \delta 2$, where $\varphi^{(2)} = (\frac 1 {\zeta^n} \varphi(\frac \cdot \zeta)) \ast \varphi^{(1)}$.} The function $\varphi^{(2)}$ is in $C^\infty(\Rn)$ with a compact support, thus is in $\scrS(\Rn)$. Write $\varphi^{(3)} = (I - \Delta)^{-s/2} \big( \agl[\cdot]^{1/2 + \epsilon} \varphi^{(2)} \big)$. Hence $\varphi^{(3)} \in \scrS(\Rn)$ and 
	\begin{align*}
	\nrm[H_{-1/2 - \epsilon}^s(\R^n)]{\varphi - \varphi^{(3)}}
	& = \nrm[L^2(\Rn)]{\agl[\cdot]^{-1/2 - \epsilon} (I - \Delta)^{s/2} \varphi - \agl[\cdot]^{-1/2 - \epsilon} (I - \Delta)^{s/2} \varphi^{(3)}} \\
	& \leq \nrm[L^2(\Rn)]{\agl[\cdot]^{-1/2 - \epsilon} (I - \Delta)^{s/2} \varphi - \varphi^{(1)}} + \nrm[L^2(\Rn)]{\varphi^{(1)} - \varphi^{(2)}} \\
	& < \delta/2 + \delta/2 = \delta.
	\end{align*}
	Therefore $\scrS(\Rn)$ is dense in $H_{-1/2 - \epsilon}^s(\R^n)$. Since $H_{-1/2 - \epsilon}^s(\R^n)$ is a Banach space, and hence by a density argument, the inequality \eqref{eq:VBdd-MLmGWsSchroEqu2019} can be extended to all $\varphi \in H_{-1/2 - \epsilon}^s(\R^n)$. 
	
	The proof is complete.
\end{proof}

We are now in a position to study the well-posedness of the direct scattering problem. To that end, we reformulate \eqref{eq:1-MLmGWsSchroEqu2019} into the Lippmann-Schwinger equation formally (cf. \cite{colton2012inverse}) to obtain
\begin{equation} \label{eq:uscDefn-MLmGWsSchroEqu2019}
(I - \Rk q) u^{sc} = \alpha \Rk q u^i - \Rk f.
\end{equation}


\begin{thm} \label{thm:MildSolUnique-MLmGWsSchroEqu2019}
	When $k$ is large enough such that $\nrm[\mathcal{L}(H_{1/2+\epsilon}^{-s}(\R^3), H_{1/2+\epsilon}^{-s}(\R^3))]{\Rk q} < 1$, there exists a unique stochastic process $u^{sc}(\cdot,\omega) \colon \R^3 \to \mathbb C$ such that $u^{sc}(x)$ satisfies \eqref{eq:uscDefn-MLmGWsSchroEqu2019} almost surely. Moreover, 
	\begin{equation} \label{eq:MildSolUnique-MLmGWsSchroEqu2019}
		\nrm[H_{1/2+\epsilon}^{-s}(\R^3)]{u^{sc}(\cdot,\omega)}
		\lesssim \nrm[H_{1/2+\epsilon}^{-s}(\R^3)]{\alpha \Rk q u^i - \Rk f} \quad a.s.
	\end{equation}
	for any $\epsilon \in \mathbb{R}_+$. 
\end{thm}

\begin{proof}
	\sq{The condition \eqref{eq:mqmf-MLmGWsSchroEqu2019} implies $m_q > 2$, and hence there exists ${s \in (\max\{(3 - m_q)/2, 0\} , 1/2)}$ such that Theorem \ref{thm:RkBounded-MLmGWsSchroEqu2019} can apply.}
	By Theorems \ref{thm:RkBounded-MLmGWsSchroEqu2019} and \ref{thm:VBounded-MLmGWsSchroEqu2019}, we know
	\[
	F := \alpha \Rk q u^i - \Rk f \in H_{1/2+\epsilon}^{-s}(\R^3).
	\]
	From Theorems \ref{thm:RkBounded-MLmGWsSchroEqu2019} and \ref{thm:VBounded-MLmGWsSchroEqu2019}, we also know that the operator $I - \Rk q$ is invertible from $H_{1/2+\epsilon}^{-s}(\R^3)$ to itself, and the right-hand side of \eqref{eq:uscDefn-MLmGWsSchroEqu2019} belongs to $H_{1/2+\epsilon}^{-s}(\R^3)$.
	
	Let $u^{sc} := (I - \Rk q)^{-1} F \in H_{1/2+\epsilon}^{-s}(\R^3)$, then $u^{sc}$ fulfills the requirements of the theorem. 
	The existence of the solution is proved.
	\eqref{eq:MildSolUnique-MLmGWsSchroEqu2019} can be verified easily from Theorems \ref{thm:RkBounded-MLmGWsSchroEqu2019}, \ref{thm:VBounded-MLmGWsSchroEqu2019} and \eqref{eq:uscDefn-MLmGWsSchroEqu2019}.
	The uniqueness follows readily from \eqref{eq:MildSolUnique-MLmGWsSchroEqu2019}.
	
	The proof is complete. 
\end{proof}

\section{Asymptotic analysis of high-order terms} \label{sec:AsympHighOrder-MLmGWsSchroEqu2019}

We intend to recover $\mu_f$, $\mu_q$ from the data via the correlation formula of the following form 
\begin{equation} \label{eq:GenForm-MLmGWsSchroEqu2019}
\frac 1 K \int_K^{2K} k^m \overline{u^\infty(k,\omega)} u^\infty (k+\tau,\omega) \dif k,
\end{equation}
where $u^\infty(k,\omega)$ stands for the far-field pattern $u^\infty(\hat x,k,\omega) \in \mathcal M_f$ in the case of $\alpha = 0$ and stands for $u^\infty(\hat x,k,-\hat x,\omega) \in \mathcal M_q$ in the case of $\alpha = 1$.
The Lippmann-Schwinger equation corresponding to \eqref{eq:1-MLmGWsSchroEqu2019} is
\begin{equation} \label{eq:lsdual-MLmGWsSchroEqu2019}
(I - \Rk q)u^{sc}(k,\omega) = \alpha \Rk q u^i \sq{-} \Rk f.
\end{equation}
When $k$ is large enough such that $\nrm[\mathcal{L} (H_{-1/2 - \epsilon}^s,H_{-1/2 - \epsilon}^s)]{\Rk q} < 1$, from \eqref{eq:lsdual-MLmGWsSchroEqu2019} we obtain
\begin{align}
u^{sc}(k,\omega) & = \sq{-\sum_{j \geq 0} \Rk \big( (q \Rk)^j f \big) + \alpha \sum_{j \geq 0} \Rk \big( (q \Rk)^j qu^i \big)},  \label{eq:scattered-MLmGWsSchroEqu2019}\\
u^\infty(k,\omega) &= (4\pi)^{-1} \sum_{j=0,1,2} F_j(\hat x,k,\omega) + \alpha (4\pi)^{-1} \sum_{j=0,1,2} G_j(\hat x,k,\omega),  \label{eq:farfieldFG-MLmGWsSchroEqu2019}
\end{align}
where
\begin{equation} \label{eq:FjGjDef-MLmGWsSchroEqu2019}
\left\{\begin{aligned}
F_j(\hat x,k,\omega) & := \sq{-}\int_{\R^3} e^{-\textrm{i}k\hat x \cdot z} \big[ (q\Rk)^j f \big](z) \dif z, \quad j = 0,1 \\
F_2(\hat x,k,\omega) & := \sq{-} \sum_{j \geq 2} \int_{\R^3} e^{-\textrm{i}k\hat x \cdot z} \big[ (q\Rk)^j f \big](z) \dif z, \\
G_j(\hat x,k,d,\omega) & := \int_{\R^3} e^{-\textrm{i}k\hat x \cdot z} \sq{\big[ (q \Rk)^j q u^i \big](z)} \dif z, \quad j = 0,1 \\
G_2(\hat x,k,d,\omega) & := \sum_{j \geq 2} \int_{\R^3} e^{-\textrm{i}k\hat x \cdot z} \sq{\big[ (q \Rk)^j q u^i \big](z)} \dif z.
\end{aligned}\right.
\end{equation}
Substituting \eqref{eq:farfieldFG-MLmGWsSchroEqu2019} into \eqref{eq:GenForm-MLmGWsSchroEqu2019}, we obtain several crossover terms comprised by $F_j$ and $G_j$. 
To recover $\mu_f$ and $\mu_q$, it is necessary to establish the asymptotics of $F_j$ and $G_j$ in terms of $k$. 
The asymptotic analyses of $G_j\, (j=0,1,2)$ are established in \cite{caro2016inverse}.

This section is devoted to the asymptotic analysis of $F_1$ and $F_2$,
which are given in Lemmas \ref{lem:HOTF1F1-MLmGWsSchroEqu2019} and \ref{lem:HOTF2F2-MLmGWsSchroEqu2019}, respectively. 
These two lemmas shall play key roles in the proofs to Theorems \ref{thm:UniSource-MLmGWsSchroEqu2019} and \ref{thm:UniPot1-MLmGWsSchroEqu2019}.

\subsection{Asymptotics of $F_1$} \label{subsec:AsyF1-MLmGWsSchroEqu2019}

In order to establish the asymptotics of $F_1$, we need to derive two auxiliary lemmas. First, let us recall the notion of the fractional Laplacian \cite{pozrikidis2016fractional} of order 
$s \in (0,1)$ in $\mathbb{R}^n$ ($n\geq 3$),
\begin{equation} \label{eq:FracLapDef-MLmGWsSchroEqu2019}
(-\Delta)^{s/2} \varphi(x) := (2\pi)^{-n} \iint e^{{\textrm{i}}(x-y) \cdot \xi} |\xi|^s \varphi(y) \dif y \dif \xi,
\end{equation}
where the integration is defined as an oscillatory integral.
When $\varphi \in \scrS(\Rn)$, \eqref{eq:FracLapDef-MLmGWsSchroEqu2019} can be understood as a usual Lebesgue integral if one integrates w.r.t.~$y$ first and then integrates w.r.t.~$\xi$.
By duality arguments, the fractional Laplacian can be generalized to act on wider range of functions and distributions (cf. \cite{wong2014pdo}).
It can be verified that the fractional Laplacian is self-adjoint. 

In the following two lemmas, we present the results in a more general form where the space dimension $n$ can be arbitrary but greater than 2, though only the case $n=3$ shall be used subsequently. 

\begin{lem} \label{lem:FracExp-MLmGWsSchroEqu2019}
	For any $s \in (0,1)$, we have
	\begin{equation*}
	(-\Delta_\xi)^{s/2} (e^{{\textrm{i}}x \cdot \xi}) = |x|^s e^{{\textrm{i}}x \cdot \xi}
	\end{equation*}
	in the distributional sense.
\end{lem}

\begin{proof}
	For any $\varphi \in \scrS(\R^n)$, because $(-\Delta_\xi)^{s/2}$ is self-adjoint, we have
	\begin{align*}
	\big( (-\Delta_\xi)^{s/2} (e^{{\textrm{i}}x \cdot \xi}), \varphi(\xi) \big)
	& = \big( e^{{\textrm{i}}x \cdot \xi}, (-\Delta_\xi)^{s/2} \varphi(\xi) \big) \\
	& = \int e^{{\textrm{i}}x \cdot \xi} \cdot \big[ (2\pi)^{-n} \iint e^{{\textrm{i}}(\xi-y) \cdot \eta} |\eta|^s \varphi(y) \dif y \dif \eta \big] \dif \xi \\
	& = \int e^{{\textrm{i}}x \cdot \xi} \cdot \big\{ (2\pi)^{-n/2} \int \big[ (2\pi)^{-n/2} \int e^{{\textrm{i}}(\xi-y) \cdot \eta} |\eta|^s \dif \eta \big] \varphi(y) \dif y \big\} \dif \xi \\
	& = (2\pi)^{-n/2} \int e^{{\textrm{i}}x \cdot \xi} \cdot \int \calF^{-1}\{ |\cdot|^s \}(\xi-y) \cdot \varphi(y) \dif y \dif \xi \\
	& = (2\pi)^{-n/2} \iint e^{{\textrm{i}}x \cdot \xi} \calF^{-1}\{ |\cdot|^s \}(\xi-y) \cdot \varphi(y) \dif y \dif \xi \\
	& = \int \big[ (2\pi)^{-n/2} \int e^{{\textrm{i}}x \cdot \xi} \calF^{-1}\{ |\cdot|^s \}(\xi-y) \dif \xi \big] \cdot \varphi(y) \dif y \\
	& = \int e^{{\textrm{i}}x \cdot y} \big[ (2\pi)^{-n/2} \int e^{-{\textrm{i}}(-x) \cdot \xi} \calF^{-1} \{ |\cdot|^s \}(\xi) \dif \xi \big] \cdot \varphi(y) \dif y \\
	& = \int e^{{\textrm{i}}x \cdot y} \calF \calF^{-1} \{ |\cdot|^s \}(-x) \cdot \varphi(y) \dif y \\
	& = \int |x|^s e^{{\textrm{i}}x \cdot y} \cdot \varphi(y) \dif y \\
	& = \big( |x|^s e^{{\textrm{i}}x \cdot \xi}, \varphi(\xi) \big).
	\end{align*}
	It is noted that in the derivation above, some integrals should be understood as oscillatory integrals.
\end{proof}

\begin{lem} \label{lem:FracSymbol-MLmGWsSchroEqu2019}
	For any $m < 0$, $s \in (0,1)$ and $c(x,\xi) \in S^m$, we have
	\[
	|\big( (-\Delta_\xi)^{s/2} c \big) (x,\xi)| \leq C \agl[\xi]^{m-s},
	\]
	where the constant $C$ is independent of $x$, $\xi$.
\end{lem}


\begin{proof}
	The proof is divided into two steps.
	
	\smallskip
	
	\noindent \textbf{Step 1}: The case $|\xi| \geq 1$.
	
	\medskip
	
	In this step, we set $|\xi|$ to be greater than 1.
	By the definition \eqref{eq:FracLapDef-MLmGWsSchroEqu2019}, we have
	\begin{align}
	\big( (-\Delta_\xi)^{s/2} c \big) (x,\xi)
	& \simeq \iint e^{{\textrm{i}}(\xi - \eta) \cdot \gamma} |\gamma|^s\, c(x,\eta) \dif \eta \dif \gamma \nonumber\\
	& = \iint e^{-{\textrm{i}}\eta \cdot \gamma} |\gamma|^s\, c(x,\eta + \xi) \dif \eta \dif \gamma \nonumber\\
	& = \iint e^{-{\textrm{i}}\eta \cdot \gamma} \big| \frac \gamma {|\xi|} \big|^s\,  c(x,|\xi|\eta + \xi) \dif (|\xi|\eta) \dif (\gamma/|\xi|) \nonumber\\
	& \simeq |\xi|^{-s} \iint e^{-{\textrm{i}}\eta \cdot \gamma} |\gamma|^s\,  c(x,|\xi|(\eta + \hat \xi)) \dif \eta \dif \gamma, \label{eq:FracSymbolInter1-MLmGWsSchroEqu2019}
	\end{align}
	where $\hat \xi = \xi / |\xi|$. 
	Fix a function $\chi_0 \in C_c^\infty(\R)$ with
	$\chi_0(|x|) \equiv 1$ when $1/2 \leq |x| \leq 3/2$
	and
	$\chi_0(|x|) \equiv 1$ when $|x| \leq 0$ or $|x| \geq 2$.
	We can continue \eqref{eq:FracSymbolInter1-MLmGWsSchroEqu2019} as
	\begin{align}
	\big( (-\Delta_\xi)^{s/2} c \big) (x,\xi)
	& \simeq |\xi|^{m-s} \iint e^{-{\textrm{i}}\eta \cdot \gamma} \chi_0(|\eta|) |\gamma|^s\, c(x,|\xi|(\eta + \hat \xi))\, |\xi|^{-m} \dif \eta \dif \gamma \nonumber\\
	& \ \ \ + |\xi|^{m-s} \iint e^{-{\textrm{i}}\eta \cdot \gamma} \big( 1 - \chi_0(|\eta|)\big)\, |\gamma|^s\, c(x,|\xi|(\eta + \hat \xi))\, |\xi|^{-m} \dif \eta \dif \gamma \nonumber\\
	& := |\xi|^{m-s} (\mathcal B_1 + \mathcal B_2). \label{eq:FracSymbolInter2-MLmGWsSchroEqu2019}
	\end{align}
	We estimate $\mathcal B_1$, $\mathcal B_2$ seperately. For $\mathcal B_1$, one can compute
	\begin{align}
	\mathcal B_1
	& = \iint e^{-{\textrm{i}}(\eta - \hat \xi) \cdot \gamma} \chi_0(|\eta - \hat \xi|) |\gamma|^s\, c(x,|\xi|\eta)\, |\xi|^{-m} \dif \eta \dif \gamma \nonumber\\
	& = \int e^{{\textrm{i}}\hat \xi \cdot \gamma} |\gamma|^s \big( \int e^{-{\textrm{i}}\eta \cdot \gamma} \chi_0(|\eta - \hat \xi|)\, c(x,|\xi|\eta)\, |\xi|^{-m} \dif \eta \big) \dif \gamma \nonumber\\
	& =: \int e^{{\textrm{i}}\hat \xi \cdot \gamma} |\gamma|^s J(\gamma;|\xi|,x) \dif \gamma, \label{eq:FracSymbolB11-MLmGWsSchroEqu2019}
	\end{align}
	where
	\(
	J(\gamma;|\xi|,x) = \int e^{-{\textrm{i}}\eta \cdot \gamma} \chi_0(|\eta - \hat \xi|)\, c(x,|\xi|\eta)\, |\xi|^{-m} \dif \eta.
	\)
	We claim that the $J(\gamma;|\xi|,x)$ is rapidly decaying w.r.t.~$|\gamma|$, that is
	\begin{equation} \label{eq:FracSymbolB14-MLmGWsSchroEqu2019}
	\forall N \in \mathbb N, \quad
	|\gamma|^{2N}\, |J(\gamma;|\xi|,x)| \leq C_N < +\infty, 
	\end{equation}
	for some constant $C_N$ independent of $\gamma$, $\xi$ and $x$.
	This can be seen from
	\begin{align}
	|\gamma|^{2N}\, |J(\gamma;|\xi|,x)|
	& \simeq \big| \int \Delta_\eta^N (e^{-{\textrm{i}}\eta \cdot \gamma}) \cdot \chi_0(|\eta - \hat \xi|)\, c(x,|\xi|\eta)\, |\xi|^{-m} \dif \eta \big| \nonumber\\
	& = \big| \int e^{-{\textrm{i}}\eta \cdot \gamma} \cdot \Delta_\eta^N \big( \chi_0(|\eta - \hat \xi|)\, c(x,|\xi|\eta) \big)\, |\xi|^{-m} \dif \eta \big| \nonumber\\
	& \leq \int_{\frac 1 2 \leq |\eta - \hat \xi| \leq 2} |\Delta_\eta^N \big( \chi_0(|\eta - \hat \xi|)\, c(x,|\xi|\eta)| \cdot |\xi|^{-m} \dif \eta \nonumber\\
	& \lesssim \int_{\frac 1 2 \leq |\eta - \hat \xi| \leq 2} \sum_{|\alpha| \leq 2N} |(\partial_\xi^\alpha c) (x,|\xi|\eta)| \cdot |\xi|^{|\alpha|-m} \dif \eta \nonumber\\
	& \lesssim \sum_{|\alpha| \leq 2N} \int_{\frac 1 2 \leq |\eta - \hat \xi| \leq 2} (1 + |\xi|\,|\eta|)^{m - |\alpha|} \cdot |\xi|^{|\alpha|-m} \dif \eta \nonumber\\
	& = \sum_{|\alpha| \leq 2N} \int_{\frac 1 2 \leq |\eta - \hat \xi| \leq 2} (|\xi|^{-1} + |\eta|)^{m - |\alpha|} \dif \eta,
	\label{eq:FracSymbolB12-MLmGWsSchroEqu2019}
	\end{align}
	where $N$ is an arbitrary non-negative integer. The condition $|\xi| \geq 1$ gives
	\begin{equation} \label{eq:FracSymbolB13-MLmGWsSchroEqu2019}
	(|\xi|^{-1} + |\eta|)^{m - |\alpha|} \leq \begin{cases}
	(1 + |\eta|)^{m - |\alpha|}, &\mbox{when}~ |\alpha| \leq m, \\
	|\eta|^{m - |\alpha|}, &\mbox{when}~ |\alpha| > m.
	\end{cases}
	\end{equation}
	By \eqref{eq:FracSymbolB12-MLmGWsSchroEqu2019} and \eqref{eq:FracSymbolB13-MLmGWsSchroEqu2019}, we obtain \eqref{eq:FracSymbolB14-MLmGWsSchroEqu2019}.
	Therefore, $J(\gamma;|\xi|,x)$ is indeed rapidly decaying.
	Now, combining \eqref{eq:FracSymbolB11-MLmGWsSchroEqu2019} and \eqref{eq:FracSymbolB14-MLmGWsSchroEqu2019}, we arrive at
	\begin{equation} \label{eq:FracSymbolB1-MLmGWsSchroEqu2019}
	|\mathcal B_1|
	\lesssim \int_{|\gamma| \geq 1} |\gamma|^s  \dif \gamma
	+
	\int_{|\gamma| > 1} |\gamma|^s |\gamma|^{-4} \dif \gamma 
	\leq C < +\infty, 
	\end{equation}
	for some constant $C$ independent of $x$, $\xi$.
	
	To estimate $\mathcal B_2$, 
	we split $\mathcal B_2$ into two terms, say, $\mathcal B_{21}$ and $\mathcal B_{22}$, in the following way,
	\begin{align}
	\mathcal B_2
	& = \iint_{\gamma \leq 1} e^{-\textrm{i}\eta \cdot \gamma} \big( 1 - \chi_0(|\eta|)\big)\, |\gamma|^s\, c(x,|\xi|(\eta + \hat \xi))\, |\xi|^{-m} \dif \eta \dif \gamma \nonumber\\
	& \ \ \ + \iint_{\gamma > 1} e^{-\textrm{i}\eta \cdot \gamma} \big( 1 - \chi_0(|\eta|)\big)\, |\gamma|^s\, c(x,|\xi|(\eta + \hat \xi))\, |\xi|^{-m} \dif \eta \dif \gamma \nonumber\\
	& =: \mathcal B_{21} + \mathcal B_{22}. \label{eq:FracSymbolB2Split-MLmGWsSchroEqu2019}
	\end{align}
	Define the differential operator $L := (\gamma/|\gamma|^2) \cdot \nabla_{\eta}$.
	The term $\mathcal B_{21}$ can be estimated as follows,
	\begin{align}
	|\mathcal B_{21}|
	& \leq \int_{|\gamma| \leq 1} |\gamma|^s \cdot \big| \int e^{-\textrm{i}\eta \cdot \gamma} \big( 1 - \chi_0(|\eta|)\big)\, c(x,|\xi|(\eta + \hat \xi))\, |\xi|^{-m} \dif \eta \big| \dif \gamma \nonumber\\
	& \simeq \int_{|\gamma| \leq 1} |\gamma|^{s} \cdot \big| \int L^{n} (e^{-\textrm{i}\eta \cdot \gamma})\, \big( 1 - \chi_0(|\eta|)\big)\, c(x,|\xi|(\eta + \hat \xi))\, |\xi|^{-m} \dif \eta \big| \dif \gamma \nonumber\\
	& \lesssim \int_{|\gamma| \leq 1} |\gamma|^{s} |\gamma|^{-n} \cdot \big| \int e^{-\textrm{i}\eta \cdot \gamma}\, \nabla_{\eta}^{n} \Big( \big( 1 - \chi_0(|\eta|)\big)\, c(x,|\xi|(\eta + \hat \xi)) \Big)\, |\xi|^{-m} \dif \eta \big| \dif \gamma \nonumber\\
	& \leq \int_{|\gamma| \leq 1} |\gamma|^{s-n} \int \Big| \nabla_\eta^{n} \Big( \big( 1 - \chi_0(|\eta|)\big)\, c(x,|\xi|(\eta + \hat \xi)) \Big) \Big| \cdot |\xi|^{-m} \dif \eta \dif \gamma \nonumber\\
	& \lesssim \int_{|\gamma| \leq 1} |\gamma|^{s-n} \int_{|\eta| \not\in (\frac 1 2, \frac 3 2)} (1+|\xi| \cdot |\eta + \hat \xi|)^{m-n} \cdot |\xi|^{n-m} \dif \eta \dif \gamma \nonumber\\
	& = \int_{|\gamma| \leq 1} |\gamma|^{s-n} \int_{|\eta| \not\in (\frac 1 2, \frac 3 2)} (|\xi|^{-1} + |\eta + \hat \xi|)^{m-n} \dif \eta \dif \gamma \nonumber\\
	& \leq \int_{|\gamma| \leq 1} |\gamma|^{s-n} \int_{|\eta| \not\in (\frac 1 2, \frac 3 2)} |\eta + \hat \xi|^{m-n} \dif \eta \dif \gamma \nonumber\\
	& \leq C < +\infty, \label{eq:FracSymbolB21-MLmGWsSchroEqu2019}
	\end{align}
	for some constant $C$ independent of $x$, $\xi$. Here, it is noted that in \eqref{eq:FracSymbolB21-MLmGWsSchroEqu2019} $n$ is the space dimension. 
	The last two inequalities in \eqref{eq:FracSymbolB21-MLmGWsSchroEqu2019} make use of the following three facts:
	$s - n > -n$,
	$m - n < -n$, and 
	{the restriction $|\eta| \not\in (1/2, 3/2)$ that makes $|\eta+\hat\xi|\geq 1/2$.}
	
	To estimate $\mathcal B_{22}$, we proceed in a way similar to \eqref{eq:FracSymbolB21-MLmGWsSchroEqu2019}, 
	but replacing $L^{n}$ with $L^{n+1}$,
	\begin{align}
	|\mathcal B_{22}|
	& \lesssim \int_{|\gamma| > 1} |\gamma|^{s-1-n} \int \Big| \nabla_\eta^{n+1} \Big( \big( 1 - \chi_0(|\eta|)\big)\, c(x,|\xi|(\eta + \hat \xi)) \Big) \Big| \cdot |\xi|^{-m} \dif \eta \dif \gamma \nonumber\\
	& \lesssim \int_{|\gamma| > 1} |\gamma|^{s-1-n} \int_{|\eta| \not\in (\frac 1 2, \frac 3 2)} (|\xi|^{-1} + |\eta + \hat \xi|)^{m-1-n} \dif \eta \dif \gamma \nonumber\\
	& \leq \int_{|\gamma| > 1} |\gamma|^{s-1-n} \int_{|\eta| \not\in (\frac 1 2, \frac 3 2)} |\eta + \hat \xi|^{m-1-n} \dif \eta \dif \gamma \nonumber\\
	& \leq C < +\infty, \label{eq:FracSymbolB22-MLmGWsSchroEqu2019}
	\end{align}
	for some constant $C$ independent of $x$, $\xi$. 
	Also, the last two inequality in \eqref{eq:FracSymbolB22-MLmGWsSchroEqu2019} take advantage of the following three facts:
	$s - 1 - n < -n$,
	$m - 1 - n < -n$, and 
	the restriction $|\eta| \not\in (1/2, 3/2)$ that makes $|\eta+\hat\xi| \geq 1/2$.
	
	Finally, by \eqref{eq:FracSymbolInter2-MLmGWsSchroEqu2019}, \eqref{eq:FracSymbolB1-MLmGWsSchroEqu2019}, \eqref{eq:FracSymbolB2Split-MLmGWsSchroEqu2019}, \eqref{eq:FracSymbolB21-MLmGWsSchroEqu2019} and \eqref{eq:FracSymbolB22-MLmGWsSchroEqu2019}, we arrive at
	\begin{equation} \label{eq:FracSymbolGt1-MLmGWsSchroEqu2019}
	|\big( (-\Delta_\xi)^{s/2} c \big) (x,\xi)| \leq C |\xi|^{m-s},
	\quad \text{for all} \quad
	|\xi| \geq 1.
	\end{equation}

	\smallskip
	
	\noindent \textbf{Step 2}: The case $|\xi| < 1$.
	
	\medskip
	
	In this step, $|\xi|$ is set to be smaller than 1. We differentiate $\big( (-\Delta_\xi)^{s/2} c \big) (x,\xi)$ formally w.r.t.~$\xi$, and follow the arguments similar to \eqref{eq:FracSymbolB21-MLmGWsSchroEqu2019}-\eqref{eq:FracSymbolB22-MLmGWsSchroEqu2019},
	\begin{align}
	|\partial_{\xi_j} \big( (-\Delta_\xi)^{s/2} c \big) (x,\xi)|
	& \simeq |\partial_{\xi_j} \iint e^{\textrm{i}(\xi - \eta) \cdot \gamma} |\gamma|^s\, c(x,\eta) \dif \eta \dif \gamma| \nonumber\\
	& \lesssim |\iint_{|\gamma| \leq 1} L^{1+n} (e^{\textrm{i}(\xi - \eta) \cdot \gamma}) |\gamma|^s \gamma_j\, c(x,\eta) \dif \eta \dif \gamma| \nonumber\\
	& \ \ \ + |\iint_{|\gamma| > 1} L^{2 +n} (e^{\textrm{i}(\xi - \eta) \cdot \gamma}) |\gamma|^s \gamma_j\, c(x,\eta) \dif \eta \dif \gamma| \nonumber\\
	& \lesssim \int_{|\gamma| \leq 1} |\gamma|^{s-n} \int \agl[\eta]^{m-1-n} \dif \eta \dif \gamma \nonumber\\
	& \ \ \ + \int_{|\gamma| > 1} |\gamma|^{s-1-n} \int \agl[\eta]^{m-2-n} \dif \eta \dif \gamma \nonumber\\
	& \leq C < +\infty, \label{eq:FracSymbolLt1Inter1-MLmGWsSchroEqu2019}
	\end{align}
	where the constant $C$ is independent of $x$ and $\xi$.
	Therefore, $\big( (-\Delta_\xi)^{s/2} c \big) (x,\xi)$ is continuous w.r.t.~$\xi$ in $\mathbb{R}^n$. 
	Moreover, the gradient w.r.t.~$x$ and $\xi$ is bounded. 
	Therefore, $\big( (-\Delta_\xi)^{s/2} c \big) (x,\xi)$ is uniformly bounded for all $x \in \Rn$ and all $|\xi| \leq 1$. 
	Combining this with \eqref{eq:FracSymbolGt1-MLmGWsSchroEqu2019}, we arrive at the conclusion.
	
	The proof is complete. 
\end{proof}

By the commutability between $(-\Delta_\xi)^{s/2}$ and differential operators, we can readily obtain the following corollary.

\begin{cor} \label{cor:FracSymbol-MLmGWsSchroEqu2019}
	For any $m < 0$ and $s \in (0,1)$, we have
	\[
	\big( (-\Delta_\xi)^{s/2} c \big) (x,\xi) \in S^{m-s}
	\quad \text{for any} \quad
	c(x,\xi) \in S^m.
	\]
\end{cor}

\begin{proof}
	Write $\tilde c(x,\xi) = (-\Delta_\xi)^{s/2} c(x,\xi)$. Then
	\begin{align*}
	\partial_x^\alpha \partial_\xi^\beta \tilde c(x,\xi)
	& \simeq \partial_x^\alpha \partial_\xi^\beta \iint e^{\textrm{i}(\xi - \eta) \cdot \gamma} |\gamma|^s\, c(x,\eta) \dif \eta \dif \gamma \nonumber\\
	& \simeq \partial_x^\alpha \partial_\xi^\beta \int e^{\textrm{i}\xi \cdot \gamma} |\gamma|^s\, \calF_{\xi \to \gamma} \{c\}(x,\gamma) \dif \gamma \nonumber\\
	& \simeq \partial_x^\alpha \int e^{\textrm{i}\xi \cdot \gamma} |\gamma|^s\, \gamma^\beta \calF_{\xi \to \gamma} \{c\}(x,\gamma) \dif \gamma \nonumber\\
	& \simeq \partial_x^\alpha \int e^{\textrm{i}\xi \cdot \gamma} |\gamma|^s\, \calF_{\xi \to \gamma} \{\partial_\xi^\beta(c)\}(x,\gamma) \dif \gamma \nonumber\\
	& \simeq \partial_x^\alpha \iint e^{\textrm{i}(\xi - \eta) \cdot \gamma} |\gamma|^s\, (\partial_\xi^\beta c)(x,\eta) \dif \eta \dif \gamma \nonumber\\
	& = \iint e^{\textrm{i}(\xi - \eta) \cdot \gamma} |\gamma|^s\, (\partial_x^\alpha \partial_\xi^\beta c)(x,\eta) \dif \eta \dif \gamma \nonumber\\
	& = \big( (-\Delta_\xi)^{s/2} (\partial_x^\alpha \partial_\xi^\beta c) \big) (x,\xi).
	\end{align*}
	Applying Lemma \ref{lem:FracSymbol-MLmGWsSchroEqu2019}, we obtain
	\[
	|\partial_x^\alpha \partial_\xi^\beta \tilde c(x,\xi)| \leq C_{\alpha,\beta} \agl[\xi]^\beta.
	\]
	The proof is complete.
\end{proof}

\bigskip

Recall the definition of the unit normal vector $\boldsymbol{n}$ {after \eqref{eq:fqSeparation-MLmGWsSchroEqu2019}.}
The asymptotic estimate associated with the term $F_1$ is established in the following lemma.

\begin{lem} \label{lem:HOTF1F1-MLmGWsSchroEqu2019}
	We have
	\begin{equation} \label{eq:hotF1F1-MLmGWsSchroEqu2019}
	\mathbb{E} (|F_1(\hat x,k,\cdot)|^2) \leq C k^{-4},\quad \Forall k > 1, 
	\end{equation}
	for all $\hat x$ with $\hat x \cdot \boldsymbol n \geq 0$,
	and the constant $C$ in \eqref{eq:hotF1F1-MLmGWsSchroEqu2019} is independent of $\hat x$, $k$.
\end{lem}


In what follows, we shall use ${\calC(\cdot)}$ and its variants, such as ${\vec \calC(\cdot)}$, ${\calC_{a,b}(\cdot)}$ etc., to represent some generic smooth scalar/vector functions, within $C_c^\infty(\R^3)$ or $C_c^\infty(\R^{3 \times 4})$, whose particular definition may change line by line.
\begin{proof}[Proof of Lemma \ref{lem:HOTF1F1-MLmGWsSchroEqu2019}]
	Using \eqref{eq:CK-MLmGWsSchroEqu2019} and \eqref{eq:KandSymbol-MLmGWsSchroEqu2019}, one can show that
	\begin{align}
	& \ \mathbb{E} (|F_1(\hat x,k,\cdot)|^2) \nonumber\\
	= & \ \mathbb E \big( \int_{\R^3} e^{-\textrm{i}k\hat x \cdot y} q(y,\cdot) \int_{\R^3} \frac {e^{\textrm{i}k|y-s|}} {4\pi|y-s|} f(s,\cdot) \dif s \dif y \nonumber\\
	& \ \ \cdot \int_{\R^3} e^{\textrm{i}k\hat x \cdot z} \overline q(z,\cdot) \int_{\R^3} \frac {e^{-\textrm{i}k|z-t|}} {4\pi|z-t|} \overline f(t,\cdot) \dif t \dif z \big) \nonumber\\
	\simeq & \int e^{-\textrm{i}k\hat x \cdot (y-z)} \frac {e^{\textrm{i}k(|y-s|-|z-t|)}} {|y-s| \cdot |z-t|} \cdot \mathbb E \big( q(y,\cdot) \overline q(z,\cdot) \big) \cdot \mathbb E \big( f(s,\cdot) \overline f(t,\cdot) \big) \dif{(s, y, t, z)} \nonumber\\
	\simeq & \int e^{\textrm{i}k\varphi(y,s,z,t)} \big( \int e^{\textrm{i}(z-y) \cdot \xi} c_q(z,\xi) \dif \xi \big) \big( \int e^{\textrm{i}(t-s) \cdot \eta} c_f(t,\xi) \dif \eta \big) \cdot \calC \cdot \dif{(s, y, t, z)}, \label{eq:hotF1F1Inter1-MLmGWsSchroEqu2019}
	\end{align}
	where $\varphi(y,s,z,t) := -\hat x \cdot (y-z) - |y-s| + |z-t|$, and the $\dif{(s, y, t, z)}$ is a short notation for $\dif s \dif y \dif t \dif z$. We omit the repeated integral symbols and the integral domain in the calculation for simplicity. 
	The term $\calC(y,z,s,t)$ in \eqref{eq:hotF1F1Inter1-MLmGWsSchroEqu2019} belongs to $C_c^\infty(\R^{3 \times 4})$ due to the fact that $q$ and $f$ are compactly supported and ${\dist(\mathcal{CH}(D_f), \mathcal{CH}(D_q)) > 0}$.

	Next we are about to differentiate the term $e^{\textrm{i}k\varphi(y,s,z,t)}$ by two differential operators, in order to obtain the decay w.r.t.~the argument $k$. To that end, we introduce the aforesaid two differential operators with $C^\infty$-smooth coefficients as follows,
	\begin{equation*}
	L_1 := \frac {(y-s) \cdot \nabla_s} {\textrm{i}k|y-s|}, 
	\quad
	L_2 = L_{2,\hat x} := \frac {\nabla_y \varphi \cdot \nabla_y} {\textrm{i}k|\nabla_y \varphi|},
	\end{equation*}
	where ${\nabla_y \varphi = \frac {s-y} {|s-y|} - \hat x}$. The operator $L_{2,\hat x}$ depends on $\hat x$ because $\nabla_y \varphi$ does. 
	Due to the fact that $y \in D_q$ while $s \in D_f$, the operator $L_1$ is well-defined.
	It can be verified there is a positive lower bound of $|\nabla_y \varphi|$ for all $\hat x\in \{\hat x \in \mathbb{S}^2 \colon \hat x \cdot \boldsymbol{n} \geq 0\}$. 
	It can also be verified that
	\[
	L_1 (e^{\textrm{i}k\varphi(y,s,z,t)}) = L_2 (e^{\textrm{i}k\varphi(y,s,z,t)}) = e^{\textrm{i}k\varphi(y,s,z,t)}.
	\]
	
	By using integration by parts, one can compute
	\begin{align}
	& \ \mathbb{E} (|F_1(\hat x,k,\cdot)|^2) \nonumber\\
	= & \int \big( L_1^2 L_2^2 \big) (e^{\textrm{i}k\varphi(y,s,z,t)})
	\cdot 
	\big( \int e^{\textrm{i}(z-y) \cdot \xi} c_q(z,\xi) \dif \xi \big) \nonumber\\
	& \ \cdot 
	\big( \int e^{\textrm{i}(t-s) \cdot \eta} c_f(t,\eta) \dif \eta \big)
	\cdot
	\calC(y,z,s,t) \dif{(s, y, t, z)} \nonumber\\
	\simeq & \ k^{-4} \int_{\mathcal D} e^{\textrm{i}k\varphi(y,s,z,t)} \big[
	\mathcal J_1\, (\mathcal K_1\, \calC + \vec{\mathcal K}_2 \cdot \vec{\calC} + \sum_{a,b=1,2,3} \mathcal K_{3;a,b}\, \calC_{a,b}) \nonumber\\
	& \ + \sum_{c=1,2,3} \mathcal J_{2;c}\, (\mathcal K_1\, \calC_c + \vec{\mathcal K}_2 \cdot \vec{\calC}_c + \sum_{a,b=1,2,3} \mathcal K_{3;a,b}\, \calC_{a,b,c}) \nonumber\\
	& \ + \sum_{a',b'=1,2,3} \mathcal J_{3;a',b'} (\mathcal K_1\, \calC_{a',b'} + \vec{\mathcal K}_2 \cdot \vec{\calC}_{a',b'} + \sum_{a,b=1,2,3} \mathcal K_{3;a,b}\, \calC_{a,b,a',b'}) \big] \dif{(s, y, t, z)}, \label{eq:hotF1F1Inter2-MLmGWsSchroEqu2019}
	\end{align}
	where the integral domain $\mathcal D \subset \R^{3 \times 4}$ is bounded and
	\begin{alignat*}{2}
	& \mathcal J_1
	:= \int e^{\textrm{i}(t-s) \cdot \eta}\, c_f(t,\eta) \dif \eta,
	& \quad
	& \mathcal K_1
	:= \int e^{\textrm{i}(z-y) \cdot \xi}\, c_q(z,\xi) \dif \xi, \\
	& \vec{\mathcal J}_2
	:= \nabla_s \int e^{\textrm{i}(t-s) \cdot \eta}\, c_f(t,\eta) \dif \eta,
	& \quad 
	& \vec{\mathcal K}_2
	:= \nabla_y \int e^{\textrm{i}(z-y) \cdot \xi}\, c_q(z,\xi) \dif \xi, \\
	& \mathcal J_{3;a,b}
	:= \partial_{s_a,s_b}^2 \int e^{\textrm{i}(t-s) \cdot \eta}\, c_f(t,\eta) \dif \eta,
	& \quad
	& \mathcal K_{3;a,b}
	:= \partial_{y_a,y_b}^2 \int e^{\textrm{i}(z-y) \cdot \xi}\, c_q(z,\xi) \dif \xi,
	\end{alignat*}
	and $\mathcal J_{2;c}$ (resp.~$\mathcal K_{2;c}$) is the $c$-th component of the vector $\vec{\mathcal J}_2$ (resp.~$\vec{\mathcal K}_2$).

	For the case with $s \neq t$, 
	these three quantities, $\mathcal J_1$, $\vec{\mathcal J}_2$ and $\mathcal J_{3;a,b}$, can be estimated as follows,
	\begin{align}
	|\mathcal J_1|
	& = |\int e^{\textrm{i}(t-s) \cdot \eta}\, c_f(t,\eta) \dif \eta|
	= |s-t|^{-2} \cdot |\int \Delta_\eta (e^{\textrm{i}(s-t) \cdot \eta})\, c_f(t,\eta) \dif \eta| \nonumber\\
	& = |s-t|^{-2} \cdot |\int e^{\textrm{i}(t-s) \cdot \eta} (\Delta_\eta c_f) (t,\eta) \dif \eta|
	\leq |s-t|^{-2} \int |(\Delta_\eta c_f) (t,\eta)| \dif \eta \nonumber\\
	& \lesssim |s-t|^{-2} \int \agl[\eta]^{-m_f-2} \dif \eta
	\lesssim |s-t|^{-2}, \label{eq:hotF1F1J1-MLmGWsSchroEqu2019}
	\end{align}
	and
	\begin{align}
	|\vec{\mathcal J}_{2;c}|
	& = |\partial_{s_c} \int e^{\textrm{i}(t-s) \cdot \eta}\, c_f(t,\eta) \dif \eta|
	= |\int e^{\textrm{i}(t-s) \cdot \eta} \cdot c_f(t,\eta) \eta_c  \dif \eta| \nonumber\\
	& = |s-t|^{-2} \cdot |\int \Delta_\eta (e^{\textrm{i}(t-s) \cdot \eta})\, c_f(t,\eta) \eta_c \dif \eta|
	= |s-t|^{-2} \cdot |\int e^{\textrm{i}(t-s) \cdot \eta} \Delta_\eta (c_f(t,\eta) \eta_c) \dif \eta| \nonumber\\
	& \lesssim |s-t|^{-2} \int \agl[\eta]^{-m_f+1-2} \dif \eta
	\lesssim |s-t|^{-2}, \label{eq:hotF1F1J2-MLmGWsSchroEqu2019}
	\end{align}
	and similarly
	\begin{align}
	\mathcal J_{3;a,b}
	& 
	\simeq \int e^{\textrm{i}(t-s) \cdot \eta} \cdot c_f(t,\eta) \eta_a \eta_b \dif \eta
	\simeq |s-t|^{-2} \int \Delta_\eta (e^{\textrm{i}(t-s) \cdot \eta}) \cdot c_f(t,\eta) \eta_a \eta_b \dif \eta \nonumber\\
	& = |s-t|^{-2} \int e^{\textrm{i}(t-s) \cdot \eta} \cdot \Delta_\eta (c_f(t,\eta) \eta_a \eta_b) \dif \eta.
	\label{eq:hotF1F1J3Inter1-MLmGWsSchroEqu2019}
	\end{align}
	\sq{Here, in deriving the last two inequalities respectively in \eqref{eq:hotF1F1J1-MLmGWsSchroEqu2019} and \eqref{eq:hotF1F1J2-MLmGWsSchroEqu2019}, we have made use of the a-priori requirement $m_f > 2$ in  \eqref{eq:mqmf-MLmGWsSchroEqu2019}; see also the discussion at the end of Remark~\ref{rem:1.1}.}

	\smallskip
	
	Now, if we further differentiate the term $e^{\textrm{i}(t-s) \cdot \eta}$ in \eqref{eq:hotF1F1J3Inter1-MLmGWsSchroEqu2019} by $\frac {\textrm{i}(s-t) \cdot} {|s-t|^2} \nabla_\eta$ and then {transfer} the operator $\nabla_\eta$ onto $\Delta_\eta (c_f(t,\eta) \eta_a \eta_b)$ by using integration by parts, we would arrive at
	\begin{equation*}
	|\mathcal J_{3;a,b}| \lesssim |s-t|^{-3} \int |\nabla_\eta \Delta_\eta (c_f(t,\eta) \eta_a \eta_b)| \dif \eta \leq |s-t|^{-3} \int \agl[\eta]^{-m_f-1} \dif \eta.
	\end{equation*}
	The term $\int \agl[\eta]^{-m_f-1} \dif \eta$ is absolutely integrable now, but the term $|s-t|^{-3}$ is not integrable at the hyperplane $s = t$ in $\R^3$.
	{To circumvent this dilemma, the fractional Laplacian can be applied as follows.}
	By using Lemmas \ref{lem:FracExp-MLmGWsSchroEqu2019} and \ref{lem:FracSymbol-MLmGWsSchroEqu2019}, we can continue \eqref{eq:hotF1F1J3Inter1-MLmGWsSchroEqu2019} as
	\begin{align}
	|\mathcal J_{3;a,b}|
	& \simeq |s-t|^{-2} \cdot \big| |s-t|^{-s} \int (-\Delta_\eta)^{s/2} (e^{\textrm{i}(t-s) \cdot \eta}) \cdot \Delta_\eta (c_f(t,\eta) \eta_j \eta_\ell) \dif \eta \big| \nonumber\\
	& = |s-t|^{-2-s} \cdot |\int e^{\textrm{i}(t-s) \cdot \eta} \cdot (-\Delta_\eta)^{s/2} \big( \Delta_\eta (c_f(t,\eta) \eta_j \eta_\ell) \big) \dif \eta| \nonumber\\
	& \lesssim |s-t|^{-2-s} \int \agl[\eta]^{-m_f+2-2-s} \dif \eta = |s-t|^{-2-s} \int \agl[\eta]^{-m_f-s} \dif \eta, \label{eq:hotF1F1J3Inter2-MLmGWsSchroEqu2019}
	\end{align}
	where the number $s$ is chosen to satisfy \sq{$\max\{0, 3 - m_f\} < s < 1$, and the existence of such a number $s$ is guaranteed by noting that $m_f > 2$}. Therefore, we have
	\begin{subequations}
		\begin{numcases}{}
		-m_f - s < -3, \label{eq:hotF1F1J3s1-MLmGWsSchroEqu2019}\\
		-2 - s > -3. \label{eq:hotF1F1J3s2-MLmGWsSchroEqu2019}
		\end{numcases}
	\end{subequations}
	Thanks to the condition \eqref{eq:hotF1F1J3s1-MLmGWsSchroEqu2019}, we can continue \eqref{eq:hotF1F1J3Inter2-MLmGWsSchroEqu2019} as
	\begin{align}
	|\mathcal J_{3;a,b} |
	& \lesssim |s-t|^{-2-s} \int \agl[\eta]^{-m_f-s} \dif \eta
	\lesssim |s-t|^{-2-s}. \label{eq:hotF1F1J3-MLmGWsSchroEqu2019}
	\end{align}
	Using similar arguments, we can also conclude that
	\begin{equation} \label{eq:hotF1F1K-MLmGWsSchroEqu2019}
	\left\{\begin{aligned}
	|\mathcal K_1|,\, |\vec{\mathcal K}_2| & \lesssim |y-z|^{-2}, \\
	|\mathcal K_{3;a,b}| & \lesssim |y-z|^{-2-s}.
	\end{aligned}\right.
	\end{equation}
	
	Combining \eqref{eq:hotF1F1Inter2-MLmGWsSchroEqu2019}, \eqref{eq:hotF1F1J1-MLmGWsSchroEqu2019}, \eqref{eq:hotF1F1J2-MLmGWsSchroEqu2019}, \eqref{eq:hotF1F1J3-MLmGWsSchroEqu2019} and \eqref{eq:hotF1F1K-MLmGWsSchroEqu2019}, we arrive at
	\begin{align}
	& \ \mathbb{E} (|F_1(\hat x,k,\cdot)|^2) \nonumber\\
	\lesssim & \ k^{-4} \int_{\mathcal D} (|\mathcal J_1| + |\vec{\mathcal J}_2| + \sum_{a',b'=1,2,3} |\mathcal J_{3;a',b'}|) \cdot (|\mathcal K_1| + |\vec{\mathcal K}_2| + \sum_{a,b=1,2,3} |\mathcal K_{3;a,b}|) \dif{(s, y, t, z)} \nonumber\\
	\lesssim & \ k^{-4} \int_{\mathcal D} |s - t|^{-2-s} \cdot |y - z|^{-2-s} \dif{(s, y, t, z)} \nonumber\\
	\lesssim & \ k^{-4} \int_{\widetilde{\mathcal D}} |s - t|^{-2-s} \dif s \dif t \cdot \int_{\widetilde{\mathcal D}} |y - z|^{-2-s} \dif y \dif z \label{eq:hotF1F1Inter3-MLmGWsSchroEqu2019}
	\end{align}
	for some sufficiently large but bounded domain $\widetilde{\mathcal D} \subset \R^{3 \times 2}$ satisfying $\mathcal D \subset \widetilde{\mathcal D} \times \widetilde{\mathcal D}$.
	Note that the integral \eqref{eq:hotF1F1Inter3-MLmGWsSchroEqu2019} should be understood as a singular integral because of the presence of the singularities occuring when $s = t$ and $y = z$.
	By \eqref{eq:hotF1F1Inter3-MLmGWsSchroEqu2019} and \eqref{eq:hotF1F1J3s2-MLmGWsSchroEqu2019}, we can finally conclude \eqref{eq:hotF1F1-MLmGWsSchroEqu2019}. 
	
	The proof is complete.
\end{proof}

\medskip

\subsection{Asymptotics of $F_2$} \label{subsec:AsyF2-MLmGWsSchroEqu2019}

The following lemma is necessary for the estimates of $F_2(\hat x,k,\omega)$.

\begin{lem} \label{lem:cutoffBdd-MLmGWsSchroEqu2019}
	Assume that $\epsilon > 0$. For $\forall s \in \R$, $\forall k \in \R$ and $\forall \hat x \in \mathbb{S}^{n-1}$, we have
	\begin{equation*} 
	\nrm[H_{-1/2-\epsilon}^s]{e^{-\textrm{i}k\hat x \cdot (\cdot)} \varphi} \leq C_{s,\varphi} \agl[k]^s, \quad \forall \varphi \in C_c^\infty(\Rn),
	\end{equation*}
	where the constant $C_{s,\varphi}$ depends on $s$ and $\varphi$, but is independent of $\hat x$, $k$.
\end{lem}
\begin{proof}
	By the Plancherel theorem and Peetre's inequality, one has
	\begin{align*}
	\nrm[H_{-1/2-\epsilon}^s]{e^{-\textrm{i}k\hat x \cdot (\cdot)} \varphi}^2
	& = \int \agl[x]^{-1-2\epsilon} |(I - \Delta)^{s/2} \big( e^{-\textrm{i}k\hat x \cdot (\cdot)} \varphi \big) (x)|^2 \dif x \nonumber\\
	& \leq \int |(I - \Delta)^{s/2} \big( e^{-\textrm{i}k\hat x \cdot (\cdot)} \varphi \big) (x)|^2 \dif x \nonumber\\
	& \simeq \int \agl[\xi]^{2s} |\calF \big\{ e^{-\textrm{i}k\hat x \cdot (\cdot)} \varphi \big\} (\xi)|^2 \dif \xi = \int \agl[\xi]^{2s} |\widehat \varphi(\xi + k\hat x)|^2 \dif \xi \nonumber\\
	& = \int \agl[\xi - k\hat x]^{2s} |\widehat \varphi(\xi)|^2 \dif \xi \leq \agl[k]^{2s} \int \agl[\xi]^{2|s|} |\widehat \varphi(\xi)|^2 \dif \xi.
	\end{align*}
	$\widehat \varphi$ is rapidly decaying because $\varphi \in C_c^\infty(\Rn)$. 
	Thus, the integral $\int \agl[\xi]^{2|s|} |\widehat \varphi(\xi)|^2 \dif \xi$ is a finite number depending on $s$, $\varphi$.
	The proof is done.
\end{proof}

\medskip

\begin{lem} \label{lem:HOTF2F2-MLmGWsSchroEqu2019}
	For every $s \in (\frac{3-m_q} 2,\frac 1 2)$, there exists a subset $\Omega_s \subset \Omega$ with $\mathbb P(\Omega_s) = 0$ such that for $\forall\, \omega \in \Omega \backslash \Omega_s$, the inequality
	\begin{equation} \label{eq:hotF2F2-MLmGWsSchroEqu2019}
	|F_2(\hat x,k,\omega)| \leq C_s(\omega) k^{5s-2}
	\end{equation}
	holds uniformly for $\forall \hat x \in \mathbb{S}^2$ and $\forall k > 1$, where $C_s(\omega)$ is finite almost surely.
\end{lem}

\begin{proof}
	\sq{First, we note that the condition \eqref{eq:mqmf-MLmGWsSchroEqu2019} implies $(3-m_q)/2 < 1/2$, and hence $(\frac{3-m_q} 2,\frac 1 2)$ is a non-empty open interval.}
	We define $\chi_q$ (resp. $\chi_f$) as a function in $C_c^\infty(\R^3)$ with $\chi_q(x) = 1$ (resp. $\chi_f(x) = 1$) for $\forall x \in \supp q$ (resp. $\forall x \in \supp f$). 
	From \eqref{eq:FjGjDef-MLmGWsSchroEqu2019}, Theorems \ref{thm:RkBounded-MLmGWsSchroEqu2019}, \ref{thm:VBounded-MLmGWsSchroEqu2019} and Lemma \ref{lem:cutoffBdd-MLmGWsSchroEqu2019}, one can compute
	\begin{align}
	|F_2(\hat x,k,\omega)|
	& \leq \sum_{j \geq 2} \big| \int_{\R^3} e^{-\textrm{i}k\hat x \cdot z} \chi_q(z) \big[ (q\Rk)^j f \big](z) \dif z \big| \nonumber\\
	& \leq \nrm[H_{-1/2-\epsilon}^s]{e^{-\textrm{i}k\hat x \cdot (\cdot)} \chi_q} \sum_{j \geq 2} \nrm[H_{1/2+\epsilon}^{-s}]{(q\Rk)^j (f \cdot \chi_q)} \nonumber\\
	& \leq C_s \cdot \agl[k]^s \cdot C_{\epsilon,s}(\omega) \sum_{j \geq 2} k^{-j(1-2s)} \nrm[H_{1/2+\epsilon}^{-s}]{f \cdot \chi_q} \nonumber\\
	& \leq C_{\epsilon,s}(\omega) \cdot \agl[k]^s \cdot k^{-2(1-2s)} \nrm[H_{1/2+\epsilon}^{-s}]{f \cdot \chi_q} \nonumber\\
	& \leq C_{\epsilon,s}(\omega) k^{5s-2} \nrm[H_{-1/2-\epsilon}^{s}]{\chi_q}, \label{eq:HOTF2F2Inter1-MLmGWsSchroEqu2019}
	\end{align}
	with a random variable $C_{\epsilon,s}(\omega)$ that is finite almost surely. 
	The last inequality in \eqref{eq:HOTF2F2Inter1-MLmGWsSchroEqu2019} utilizes the fact that $f(\cdot,\omega)$ is microlocally isotropic of order $m_f$ so that Theorem \ref{thm:VBounded-MLmGWsSchroEqu2019} holds for $f(\cdot,\omega)$. 
	Let $\epsilon = 1/2$ in \eqref{eq:HOTF2F2Inter1-MLmGWsSchroEqu2019}, we arrive at \eqref{eq:hotF2F2-MLmGWsSchroEqu2019}. 
	
	The proof is complete. 
\end{proof}

\section{Recovery of the source} \label{sec:recSource-MLmGWsSchroEqu2019}

In this section, we focus on the recovery of $\mu_f(x)$ associated with the random source term. In the recovering procedure, only a single realization of the passive scattering measurement is used. 
Thus, $\alpha$ in \eqref{eq:1-MLmGWsSchroEqu2019} is set to be 0, and
the random sample $\omega$ is fixed. The data set $\mathcal M_f(\omega)$ is used to achieve the unique recovery.

\smallskip

We first present the following auxiliary lemma. 

\begin{lem} \label{lem:intConv-MLmGWsSchroEqu2019}
	For any stochastic process $\{g(k,\omega)\}_{k \in \R_+}$ satisfying
	\[
	\int_1^{+\infty} k^{m-1} \mathbb E (|g(k,\cdot)|) \dif k < +\infty,
	\]
	it holds that
	\[
	\lim_{K \to +\infty} \frac 1 K \int_K^{2K} k^m g(k,\omega) \dif k = 0, \ \ass \omega \in \Omega.
	\]
\end{lem}

\begin{proof}
	By $\int_1^{+\infty} k^{m-1} \mathbb E (|g(k,\cdot)|) \dif k < +\infty$ and Fubini's Theorem, we know
	\begin{equation} \label{eq:Dominate-MLmGWsSchroEqu2019}
	\int_1^{+\infty} k^{m-1} |g(k,\omega)| \dif k < +\infty, \quad \ass \omega \in \Omega,
	\end{equation}
	\sq{which implies that $g(k,\omega)$ is almost everywhere finite in terms of $k$.}
	Now we define a function \(g_K(k,w) := \frac {\chi_{(K,2K)}(k)} {2K} k^m g(k,\omega)\), where $\chi_{(K,2K)}(k)$ is the characteristic function of the interval $(K,2K)$. 
	For almost surely every fixed $\omega$, we have
	\[
	\sq{\lim_{K \to + \infty} g_K(k,\omega) = 0 \quad \aee \quad k \in [1,+\infty).}
	\]
	Moreover, the function series $\{g_K(k,\omega)\}_{K}$ is dominated, in the argument $k$, by the function $k^{m-1} g(k,\omega)$. 
	Thus, from \eqref{eq:Dominate-MLmGWsSchroEqu2019} and the dominated convergence theorem, we can conclude
	\[
	\lim_{K \to +\infty} \int_1^{+\infty} g_K(k,\omega) \dif k = 0 \ \ass \omega \in \Omega.
	\]
	The proof is complete.
\end{proof}

\medskip

We are ready to establish the recovery of $\mu_f(x)$.

\begin{proof}[Proof of Theorem \ref{thm:UniSource-MLmGWsSchroEqu2019}]
	This proof depends on Lemma \ref{lem:HOTF1F1-MLmGWsSchroEqu2019}, which requires $\hat x \cdot \boldsymbol{n} \geq 0$. Hence, we assume that $\hat x \cdot \boldsymbol{n} \geq 0$ unless otherwise stated.
	
	Recall the definition of $F_p~(p=0,1,2)$ in \eqref{eq:FjGjDef-MLmGWsSchroEqu2019}. 
	As already mentioned at the beginning of Section \ref{sec:AsympHighOrder-MLmGWsSchroEqu2019}, we correlate the data in the following form
	\begin{align}
	& \ \frac 1 K \int_K^{2K} k^{m_f} 16\pi^2 \overline{u^\infty(\hat x,k,\omega)} u^\infty (\hat x,k+\tau,\omega) \dif k \nonumber\\
	= \ & \ \sum_{p,q=0}^2 \frac 1 K \int_K^{2K} k^{m_f} \overline{F_p(\hat x,k,\omega)} F_q(\hat x,k+\tau,\omega) \dif k \nonumber\\
	=: & \sum_{p,q=0}^2 I_{p,q}(\hat x,K,\tau,\omega). \label{eq:IpqDef-MLmGWsSchroEqu2019}
	\end{align}
	According to Corollary 4.4 in \cite{caro2016inverse}, for $\forall \tau \geq 0$ and $\forall \hat x \in \mathbb{S}^2$, there exists $\Omega_{\tau,\hat x}^{0,0} \subset \Omega$, with $\mathbb P(\Omega_{\tau,\hat x}^{0,0}) = 0$, such that
	\begin{equation} \label{eq:I00-MLmGWsSchroEqu2019}
	\forall \omega \in \Omega \backslash \Omega_{\tau,\hat x}^{0,0}, \quad \lim_{K \to +\infty} I_{0,0}(\hat x,K,\tau,\omega) = (2\pi)^{3/2} \widehat \mu_f(\tau \hat x),
	\end{equation}
	which also implies that
	\begin{equation} \label{eq:Ipq1-MLmGWsSchroEqu2019}
	\forall \omega \in \Omega \backslash \Omega_{\tau,\hat x}^{0,0}, \quad \lim_{K \to +\infty} \frac 1 K \int_K^{2K} k^{m_f} |F_0(\hat x,k,\omega)|^2 \dif k = (2\pi)^{3/2} \widehat \mu_f(0).
	\end{equation}

	\smallskip
	
	We next estimate the higher order terms.
	The Cauchy-Schwarz inequality yields
	\begin{equation} \label{eq:IpqCS-MLmGWsSchroEqu2019}
	|I_{p,q}| 
	\leq \big( \frac 1 K \int_K^{2K} k^{m_f} |F_p(\hat x,k,\omega)|^2 \dif k \big)^{\frac 1 2} \cdot \big( \frac 1 K \int_K^{2K} k^{m_f} |F_q(\hat x,k+\tau,\omega)|^2 \dif k \big)^{\frac 1 2}.
	\end{equation}
	Recall that $m_f < 3$. 
	From the condition \eqref{eq:mqmf-MLmGWsSchroEqu2019} and Lemma \ref{lem:HOTF1F1-MLmGWsSchroEqu2019} we have
	\begin{equation} \label{eq:F1IntFinit-MLmGWsSchroEqu2019}
	\int_1^{+\infty} k^{m_f-1} \mathbb E (|F_1(\hat x,k,\cdot)|^2) \dif k \lesssim \int_1^{+\infty} k^{m_f-1} k^{-4} \dif k < +\infty.
	\end{equation}
	By \eqref{eq:F1IntFinit-MLmGWsSchroEqu2019} and Lemma \ref{lem:intConv-MLmGWsSchroEqu2019}, we conclude that
	\begin{equation} \label{eq:F1IntConv-MLmGWsSchroEqu2019}
	\lim_{K \to +\infty} \frac 1 K \int_K^{2K} k^{m_f} |F_1(\hat x,k,\omega)|^2 \dif k = 0 \quad \ass \omega \in \Omega.
	\end{equation}
	For every $s \in ((3-m_q)/2,1/2)$, Lemma \ref{lem:HOTF2F2-MLmGWsSchroEqu2019} gives
	\begin{equation} \label{eq:F2IntConvInter-MLmGWsSchroEqu2019}
	\frac 1 K \int_K^{2K} k^{m_f} |F_2(\hat x,k,\omega)|^2 \dif k \leq \frac {C_s(\omega)} K \int_K^{2K} k^{m_f} k^{2(5s-2)} \dif k \leq \frac {C_s(\omega)} {K^{4-m_f-10s}}.
	\end{equation}
	\sq{Recalling the condition \eqref{eq:mqmf-MLmGWsSchroEqu2019},} we know $(3-m_q)/2 < (4 - m_f)/{10}$.
	
	Choosing any $s \in \big( (3-m_q)/2, (4 - m_f)/{10} \big)$, we have ${4 - m_f - 10s > 0}$. Combining this with \eqref{eq:F2IntConvInter-MLmGWsSchroEqu2019}, we conclude that
	\begin{equation} \label{eq:F2IntConv-MLmGWsSchroEqu2019}
	\lim_{K \to +\infty} \frac 1 K \int_K^{2K} k^{m_f} |F_2(\hat x,k,\omega)|^2 \dif k = 0 \quad \ass \omega \in \Omega.
	\end{equation}
	Formula \eqref{eq:F2IntConv-MLmGWsSchroEqu2019} easily implies that
	\begin{equation} \label{eq:F2IntConvTau-MLmGWsSchroEqu2019}
	\lim_{K \to +\infty} \frac 1 K \int_K^{2K} k^{m_f} |F_2(\hat x,k+\tau,\omega)|^2 \dif k = 0 \quad \ass \omega \in \Omega,
	\end{equation}
	for every fixed $\tau \in \R$.
	
	Write $\mathcal A := \{ (p,q) \,;\, 0 \leq p,q \leq 2 \} \backslash \{(0,0)\}$.
	By \eqref{eq:IpqCS-MLmGWsSchroEqu2019}, \eqref{eq:Ipq1-MLmGWsSchroEqu2019}, \eqref{eq:F1IntConv-MLmGWsSchroEqu2019} and \eqref{eq:F2IntConvTau-MLmGWsSchroEqu2019} we have that, for $\forall \tau \geq 0$ and $\forall \hat x \in \mathbb{S}^2$ there exists $\Omega_{\tau,\hat x}^{p,q} \subset \Omega : \mathbb P(\Omega_{\tau,\hat x}^{p,q}) = 0$, $\Omega_{\tau,\hat x}^{p,q}$ depending on $\tau$ and $\hat x$, such that
	\begin{equation} \label{eq:Ipq-MLmGWsSchroEqu2019}
	\forall (p,q) \in \mathcal A, \quad \forall \omega \in \Omega \backslash \Omega_{\tau,\hat x}^{p,q}, \quad \lim_{K \to +\infty} I_{p,q}(\hat x,K,\tau,\omega) = 0.
	\end{equation}
	Write $\Omega_{\tau \hat x} := \cup_{(p,q) \in \mathcal A \cup \{(0,0)\}} \Omega_{\tau, \hat x}^{p,q}$, thus $\mathbb P(\Omega_{\tau \hat x}) = 0$. Then \eqref{eq:Ipq-MLmGWsSchroEqu2019} gives	
	\begin{equation} \label{eq:IpqOmega-MLmGWsSchroEqu2019}
	\forall \omega \in \Omega \backslash \Omega_{\tau \hat x}, \quad \forall (p,q) \in \mathcal A, \quad \lim_{K \to +\infty} I_{p,q}(\hat x,K,\tau,\omega) = 0.
	\end{equation}
	Combining \eqref{eq:IpqDef-MLmGWsSchroEqu2019}, \eqref{eq:I00-MLmGWsSchroEqu2019} and \eqref{eq:IpqOmega-MLmGWsSchroEqu2019}, we arrive at the following statement:
	\begin{equation} \label{eq:recmu-MLmGWsSchroEqu2019}
	\begin{split}
	& \forall\, y \in \R^3, \Exists \Omega_y \subset \Omega \colon \mathbb P(\Omega_y) = 0, \st
	\forall \omega \in \Omega \backslash \Omega_y, \text{ we have} \\
	& \lim_{K \to +\infty} \frac 1 K \int_K^{2K} k^{m_f} 16\pi^2 \overline{u^\infty(\hat x,k,\omega)} u^\infty(\hat x,k+\tau,\omega) \dif k = (2\pi)^{3/2} \widehat \mu_f(\tau \hat x).
	\end{split}
	\end{equation}
	
	To prove Theorem \ref{thm:UniSource-MLmGWsSchroEqu2019}, the logical order between $y$ and $\omega$ should be exchanged. 
	Denote the usual Lebesgue measure on $\R^3$ as $\mathbb L$ and the product measure $\mathbb L \times \mathbb P$ as $\mu$, and construct the product measure space $\mathbb M := (\R^3 \times \Omega, \mathcal G, \mu)$ in the canonical way, where $\mathcal G$ is the corresponding complete $\sigma$-algebra. 
	Define
	$$Z(y, \omega) := \lim_{K \to +\infty} \frac 1 K \int_K^{2K} k^{m_f} 16\pi^2 \overline{u^\infty(\hat y,k,\omega)} u^\infty(\hat y,k+|y|,\omega) \dif k - (2\pi)^{3/2} \widehat \mu_f(y).$$
	Write $\mathcal{A} := \{ (y,\omega) \in \R^3 \times \Omega \,;\, Z(y, \omega) \neq 0 \}$. Then $\mathcal{A}$ is a subset of $\mathbb M$. 
	Set $\chi_\mathcal{A}$ as the characteristic function of $\mathcal{A}$ in $\mathbb M$. By \eqref{eq:recmu-MLmGWsSchroEqu2019} we obtain
	\begin{equation} \label{eq:FubiniEq0-MLmGWsSchroEqu2019}
	\int_{R^3} \big( \int_\Omega \chi_{\mathcal A}(y,\omega) \dif{\mathbb P(\omega)} \big) \dif{\mathbb L(y)} = 0.
	\end{equation}
	By (\ref{eq:FubiniEq0-MLmGWsSchroEqu2019}) and Corollary 7 in Section 20.1 in \cite{royden2000real}, we obtain
	\begin{equation} \label{eq:FubiniEq1-MLmGWsSchroEqu2019}
	\int_{\mathbb M} \chi_{\mathcal A}(y,\omega) \dif{\mathbb \mu} = \int_\Omega \big( \int_{R^3} \chi_{\mathcal A}(y,\omega) \dif{\mathbb L(y)} \big) \dif{\mathbb P(\omega)} = 0.
	\end{equation}
	Since $\chi_{\mathcal A}(y,\omega)$ is nonnegative, (\ref{eq:FubiniEq1-MLmGWsSchroEqu2019}) implies
	\begin{equation} \label{eq:FubiniEq2-MLmGWsSchroEqu2019}
	\Exists \Omega_0 \colon \mathbb P (\Omega_0) = 0, \st \forall\, \omega \in \Omega \backslash \Omega_0,\, \int_{R^3} \chi_{\mathcal A}(y,\omega) \dif{\mathbb L(y)} = 0.
	\end{equation}
	Formula \eqref{eq:FubiniEq2-MLmGWsSchroEqu2019} further implies for every $\omega \in \Omega \backslash \Omega_0$,
	\begin{equation} \label{eq:FubiniEq3-MLmGWsSchroEqu2019}
	\Exists S_\omega \subset \R^3 \colon \mathbb L (S_\omega) = 0, \st \forall\, y \in \R^3 \backslash S_\omega, \ Z(y,\omega) = 0.
	\end{equation}
	Now Theorem \ref{thm:UniSource-MLmGWsSchroEqu2019} is proved by \eqref{eq:FubiniEq3-MLmGWsSchroEqu2019} for the case where $\hat x \cdot \boldsymbol{n} \geq 0$.
	
	Note that $\mu_f$ is real-valued, and hence
	\(
	\widehat \mu_f(\tau \hat x) = \overline{\widehat \mu_f(-\tau \hat x)}
	\)
	when $\hat x \cdot \boldsymbol{n} < 0$.
	
	The proof is complete.
\end{proof}

\section{Recovery of the potential} \label{sec:recPotential-MLmGWsSchroEqu2019}

This section is devoted to the recovery of $\mu_q(x)$ associated with the the random potential. 
The data set $\mathcal M_q(\omega)$ is utilized to achieve the recovery. Throughout this section, $\alpha$ in \eqref{eq:1-MLmGWsSchroEqu2019} is set to be 1.

\begin{proof}[Proof of Theorem \ref{thm:UniPot1-MLmGWsSchroEqu2019}]
	Similar to the proof of Theorem \ref{thm:UniSource-MLmGWsSchroEqu2019}, the case where $\hat x \cdot \boldsymbol{n} < 0$ can be proved by utilizing the fact that $\mu_q$ is real-valued. In what follows, we assume that $\hat x \cdot \boldsymbol{n} \geq 0$ unless otherwise stated.
	
	From \eqref{eq:farfieldFG-MLmGWsSchroEqu2019} we have
	\begin{align}
	& \ \frac 1 K \int_K^{2K} k^{m_q} 16\pi^2 \overline{u^\infty(\hat x,k,-\hat x,\omega)} u^\infty(\hat x,k+\tau,-\hat x,\omega) \dif k \nonumber\\
	= \ & \sum_{p,q=0}^2 \frac 1 K \int_K^{2K} k^{m_q} \sum_{p=0}^2 [ \overline{F_p(\hat x,k,\omega)} + \overline{G_p(\hat x,k,\omega)}] \cdot \sum_{q=0}^2 [F_q(\hat x,k+\tau,\omega) + G_q(\hat x,k+\tau,\omega)] \dif k \nonumber\\
	=: & \sum_{p,q=0,1,2} \big[ I'_{p,q}(\hat x,K,\tau,\omega) + J_{p,q}(\hat x,K,\tau,\omega) + L_{p,q}^1(\hat x,K,\tau,\omega) + L_{p,q}^2(\hat x,K,\tau,\omega) \big], \label{eq:IJpqDef-MLmGWsSchroEqu2019}
	\end{align}
	where
	\begin{equation} \label{eq:JLpq-MLmGWsSchroEqu2019}
	\left\{\begin{aligned}
	I'_{p,q}(\hat x,K,\tau,\omega) & := \frac 1 K \int_K^{2K} k^{m_q} \overline{F_p(\hat x,k,\omega)} F_q(\hat x,k+\tau,\omega) \dif k, \\
	J_{p,q}(\hat x,K,\tau,\omega) & := \frac 1 K \int_K^{2K} k^{m_q} \overline{G_p(\hat x,k,\omega)} G_q(\hat x,k+\tau,\omega) \dif k, \\
	L_{p,q}^1(\hat x,K,\tau,\omega) & := \frac 1 K \int_K^{2K} k^{m_q} \overline{F_p(\hat x,k,\omega)} G_q(\hat x,k+\tau,\omega) \dif k, \\
	L_{p,q}^2(\hat x,K,\tau,\omega) & := \frac 1 K \int_K^{2K} k^{m_q} \overline{G_p(\hat x,k,\omega)} F_q(\hat x,k+\tau,\omega) \dif k.
	\end{aligned}\right.
	\end{equation}
	Note that $I'_{p,q}$ differs from $I_{p,q}$, defined in \eqref{eq:IpqDef-MLmGWsSchroEqu2019}, in that the power of $k$ in the definition of $I'_{p,q}$ is $m_q$ while that of $I_{p,q}$ is $m_f$.
	
	\smallskip
	
	It is shown in \cite{caro2016inverse} that there exists $\Omega_J \subset \Omega \colon \mathbb P(\Omega_J) = 0$ such that
	\begin{align}
	\forall \omega \in \Omega \backslash \Omega_J, \quad \lim_{K \to +\infty} J_{0,0}(\hat x,K,\tau,\omega) 
	& = (2\pi)^{3/2} \widehat \mu_q(2\tau \hat x), \label{eq:J0-MLmGWsSchroEqu2019} \\
	\forall \omega \in \Omega \backslash \Omega_J, \quad \lim_{K \to +\infty} J_{p,q}(\hat x,K,\tau,\omega) & = 0, \quad (p,q) \in \mathcal A. \label{eq:Jpq-MLmGWsSchroEqu2019}
	\end{align}
	
	We conclude that there exists $\Omega_{I'} \subset \Omega \colon \mathbb P(\Omega_{I'}) = 0$ such that
	\begin{equation} \label{eq:I'-MLmGWsSchroEqu2019}
	\forall \omega \in \Omega \backslash \Omega_{I'}, \quad \lim_{K \to +\infty} \sum_{p,q=0}^2 I'_{p,q}(\hat x,K,\tau,\omega) = 0.
	\end{equation}
	The reason for \eqref{eq:I'-MLmGWsSchroEqu2019} to hold is that
	\begin{align}
	\big| \sum_{p,q=0}^2 I'_{p,q}(\hat x,K,\tau,\omega) \big|
	& \leq \frac 1 {K^{m_f - m_q}} \sum_{p,q=0}^2 \Big[ \big( \frac 1 K \int_K^{2K} k^{m_f} |F_p(\hat x,k,\omega)|^2 \dif k \big)^{\frac 1 2} \nonumber\\
	& \ \ \ \cdot \big( \frac 1 K \int_K^{2K} k^{m_f} |F_p(\hat x,k+\tau,\omega)|^2 \dif k \big)^{\frac 1 2} \Big]. \label{eq:I'pq-MLmGWsSchroEqu2019}
	\end{align}
	By \eqref{eq:Ipq1-MLmGWsSchroEqu2019}, \eqref{eq:F1IntConv-MLmGWsSchroEqu2019} and \eqref{eq:F2IntConv-MLmGWsSchroEqu2019}-\eqref{eq:F2IntConvTau-MLmGWsSchroEqu2019}, as well as a similar argument that exchanges the logical order between $\omega$ and $y$,
	we can prove that there exists $\Omega_0 \colon \mathbb P(\Omega_0) = 0$ such that for every $\omega \in \Omega \backslash \Omega_0$, one can find $S_\omega \subset \R^3 \colon \mathbb L (S_\omega) = 0$ fulfilling that for $\forall y \in \R^3 \backslash S_\omega$, there holds
	\begin{subequations} \label{eq:FFbiniEq-MLmGWsSchroEqu2019}
		\begin{numcases}{}
		\lim_{K \to +\infty} \frac 1 K \int_K^{2K} k^{m_f} |F_0(\hat y,k,\omega)|^2 \dif k = (2\pi)^{3/2} \widehat \mu_f(0), \label{eq:FFbiniEq-a-MLmGWsSchroEqu2019} \medskip\\
		\lim_{K \to +\infty} \frac 1 K \int_K^{2K} k^{m_f} |F_j(\hat y,k,\omega)|^2 \dif k = 0, \quad (j = 1,2), \label{eq:FFbiniEq-b-MLmGWsSchroEqu2019} \medskip\\
		\lim_{K \to +\infty} \frac 1 K \int_K^{2K} k^{m_f} |F_2(\hat y,k + |y|,\omega)|^2 \dif k = 0. \label{eq:FFbiniEq-c-MLmGWsSchroEqu2019}
		\end{numcases}
	\end{subequations}
	Combining \eqref{eq:I'pq-MLmGWsSchroEqu2019}-\eqref{eq:FFbiniEq-MLmGWsSchroEqu2019}, we arrive at \eqref{eq:I'-MLmGWsSchroEqu2019}.
	
	\smallskip
	
	We next analyze $\sum_{p,q=0}^2 L_{p,q}^1(\hat x,K,\tau,\omega)$,
	\begin{align}
	\big| \sum_{p,q=0}^2 L_{p,q}^1(\hat x,K,\tau,\omega) \big|
	& \leq \frac 1 {K^{m_f - m_q}} \sum_{p,q=0}^2 \Big[ \big( \frac 1 K \int_K^{2K} k^{m_f} |F_p(\hat x,k,\omega)|^2 \dif k \big)^{\frac 1 2} \nonumber\\ 
	& \ \ \ \cdot \big( \frac 1 K \int_K^{2K} k^{m_f} |G_p(\hat x,k+\tau,\omega)|^2 \dif k \big)^{\frac 1 2} \Big]. \label{eq:Lpq1Inter-MLmGWsSchroEqu2019}
	\end{align}
	By \eqref{eq:JLpq-MLmGWsSchroEqu2019}-\eqref{eq:Jpq-MLmGWsSchroEqu2019}, \eqref{eq:FFbiniEq-MLmGWsSchroEqu2019} and \eqref{eq:Lpq1Inter-MLmGWsSchroEqu2019} \sq{and the a-priori requirement $m_q < m_f$}, we conclude that
	\begin{equation} \label{eq:Lpq1-MLmGWsSchroEqu2019}
		\lim_{K \to +\infty} \big| \sum_{p,q=0}^2 L_{p,q}^1(\hat x,K,\tau,\omega) \big|
		\lesssim
		\lim_{K \to +\infty} \frac {|\widehat \mu_f(0)|} {K^{m_f - m_q}} = 0, \quad \ass
	\end{equation}
	Similarly, we can show
	\begin{equation} \label{eq:Lpq2-MLmGWsSchroEqu2019}
	\lim_{K \to +\infty} \big| \sum_{p,q=0}^2 L_{p,q}^2(\hat x,K,\tau,\omega) \big| = 0, \quad \ass
	\end{equation}
	Combining \eqref{eq:IJpqDef-MLmGWsSchroEqu2019}, \eqref{eq:J0-MLmGWsSchroEqu2019}-\eqref{eq:I'-MLmGWsSchroEqu2019} and \eqref{eq:Lpq1-MLmGWsSchroEqu2019}-\eqref{eq:Lpq2-MLmGWsSchroEqu2019}, we arrive at
	\begin{equation*}
	\lim_{K \to +\infty} \frac 1 K \int_K^{2K} k^{m_q} 16\pi^2 \overline{u^\infty(\hat x,k,-\hat x,\omega)} u^\infty(\hat x,k+\tau,-\hat x,\omega) \dif k = (2\pi)^{3/2} \widehat \mu_q(2\tau \hat x).
	\end{equation*}
	
	The proof is complete.
\end{proof}

%

\section*{Acknowledgements}

The work of J.~Li was partially supported by the NSF of China under the grant No.~11571161 and 11731006, the Shenzhen Sci-Tech Fund No.~JCYJ20170818153840322. 
The work of H.~Liu was partially supported by Hong Kong RGC general research funds, No.~12302017,  No.~12301218 and No.~12302919. The authors would like to thank the anonymous referee for many insightful and constructive comments and suggestions, which have led to significant improvements on the results as well as the presentation of the paper.

\end{document}